\documentclass[11pt]{amsart}
\calclayout

\usepackage{graphicx}
\usepackage[T1]{fontenc}
\usepackage{amsmath}
\usepackage{amssymb}
\usepackage{setspace}
\usepackage{amsthm}
\usepackage[english]{babel}
\usepackage{enumitem}
\usepackage{tikz}
\usetikzlibrary{
	hobby,
	shapes.geometric,
	calc,
	arrows,
	topaths,
	positioning
}
\usepackage{mathrsfs}
\usepackage{mathtools}
\usepackage{amsrefs}
\usepackage{xcolor}
\usepackage{float}

\usepackage[utf8]{inputenc}
\usepackage{color,soul,soulutf8}

\DeclareMathOperator{\im}{im}

\DeclareMathOperator{\Obj}{Obj}

\DeclareMathOperator{\id}{id}

\DeclareMathOperator{\length}{length}

\newcommand{\y}{\mathbf{y}}

\newcommand{\bZ}{\mathbb{Z}}

\newtheorem{theorem}{Theorem}[section]
\newtheorem*{maintheorem}{Main Theorem}
\newtheorem*{mainapplication}{Main Application}

\newtheorem{proposition}[theorem]{Proposition}
\newtheorem{definition}[theorem]{Definition}
\newtheorem{lemma}[theorem]{Lemma}

\theoremstyle{remark}
\newtheorem*{remark}{Remark}
\newtheorem*{claim}{Claim}
\newtheorem{example}[theorem]{Example}

\makeatletter
\@namedef{subjclassname@2020}{
  \textup{2020} Mathematics Subject Classification}
\makeatother

\begin{document}

\title{A Cohomological Bundle Theory for Sheaf Cohomology}
\author{Mihail Hurmuzov}
\address{Department of Mathematics, University of York, York YO10 5DD, England}
\email{mihail.hurmuzov@york.ac.uk}
\thanks{The author would like to thank Brent Everitt for suggesting the direction for this work and for numerous fruitful conversations thereafter; and Paul Turner for pointing towards the main application.}
\begin{abstract}
    We develop a bundle theory of presheaves on small categories, based on similar work by Brent Everitt and Paul Turner. For a certain set of presheaves on posets, we produce a Leray-Serre type spectral sequence that gives a reduction property for the cohomology of the presheaf. This extends the usual cohomological reduction of posets with a unique maximum.
\end{abstract}
\maketitle

\section*{Introduction}

In \cite{ET12} Everitt and Turner develop a bundle theory for coloured posets. In their development, coloured posets act as a generalisation of Khovanov's `cube' construction in his celebrated paper on the categorification of the Jones polynomial \cite{khovanov}. The notion of homology for coloured posets differs from the usual definition of sheaf homology, so it is desirable to rebuild the theory in a more versatile form. In this paper, we do away with coloured posets and move to the full generality of presheaves of modules on small categories. Additionally, we rework the arguments from the more natural \emph{co}homological point of view. For the main theoretical result we impose the general assumption that the base $\mathbf{B}$ of our bundle $\xi:\mathbf{B}\to\mathbf{Sh}$ is a poset category and that for each $x\in\mathbf{B}$, the small category of $\xi(x)$ is also a poset category; we call such a bundle a \emph{poset bundle of sheaves}:

\begin{maintheorem}
	Let $\xi:\mathbf{B}\to \mathbf{Sh}$ be a poset bundle of sheaves with $\mathbf{B}$ a recursively admissible finite poset, and $(\mathbf{E}_\xi,F_\xi)$ the associated total sheaf. Then there is a spectral sequence that converges to the cohomology of the total sheaf:
	\[
	E_2^{p,q}=H^p(\mathbf{B};\mathcal{H}^q_{fib}(\xi))\Rightarrow H^\bullet(\mathbf{E}_\xi;F_\xi).
	\]
\end{maintheorem}

Bundles of sheaves and the total sheaf are defined in \S \ref{sec:cat}, while recursively admissible posets are defined in \S \ref{sec:technical}.

In the context of sheaf cohomology, the spectral sequence for a poset bundle of sheaves converges to the cohomology of the fiber at the maximum of the base. Thus, while the main theorem of \cite{ET12} is able to model Khovanov homology, our key application is as follows.

\begin{mainapplication}
	Let $\mathbf{E}$ and $\mathbf{B}$ be posets, with $\mathbf{B}$ recursively admissible. Suppose that $\pi:\mathbf{E}\to\mathbf{B}$ is an onto poset map such that for all $x<y$ in $\mathbf{B}$, the subposet $\pi^{-1}(x)\cup\pi^{-1}(y)$ of $\mathbf{E}$ is admissible for $\pi^{-1}(x),\pi^{-1}(y)$. Then
	\[
	H^\bullet(\mathbf{E};F)\cong H^\bullet(\pi^{-1}(1);F),
	\]
	for all $F\in\mathbf{Sh}(\mathbf{E})$, where $1$ is the unique maximum of $\mathbf{B}$.
\end{mainapplication}

We proceed as follows. In \S \ref{sec:cat} we define a category $\mathbf{Sh}$ of (pre)sheaves on small categories that features morphisms between objects reminiscent of the induced maps in \cite[p.4]{ET15}. A bundle is then just a small category decorated with objects and morphisms of $\mathbf{Sh}$. We also give a way to `glue up' the elements of $\mathbf{Sh}$ in a bundle into a \emph{total sheaf}. The main aim of the paper is to understand the relationship between this total sheaf and the bundle. In \S \ref{sec:cochain} we describe explicitly the cochain modules giving rise to the sheaf cohomology of a sheaf. This gives the concrete tools needed to establish the quasi-isomorphisms in the main theorem. Next, the general construction of a spectral sequence from a bicomplex provides the first step of the overarching argument -- a spectral sequence, constructed from the bundle, that converges to a particular cohomology. The rest of the argument is establishing that this particular cohomology coincides (in some restricted cases) with the usual sheaf cohomology of the total sheaf.

The chain map $\omega$ that will witness this coincidence is defined in \S \ref{sec:traversals}. Similarly to \cite{ET12}, it involves signed combinations of traversals of a grid, determined by a pair of sequences in the base small category and in one of the small categories over an object of the base. The cohomological viewpoint here necessitates our $\omega$ goes `the opposite way' to that in \cite{ET12}. The next section collects some technical tools and the definition of a recursively admissible poset -- the restricted case in which the main theorem holds. The bulk of the work is in \S \ref{sec:long1} and \S \ref{sec:long2} with the establishment of two explicit quasi-isomorphisms, giving rise to two long exact sequence in the cohomologies of the total complex and of the sheaf. This is done by careful manipulation of spectral sequences and morphisms between them. Section \S\ref{sec:main} collects the results into a proof of the main theorem.

The final section \S\ref{sec:application} gives a recipe for turning a sheaf on a poset into the total sheaf of a bundle. If the base of that bundle is recursively admissible, then applying our main theorem completes the proof of the main application. We finish with an example of a repeated use of this application.

\section{The category $\mathbf{Sh}$}\label{sec:cat}
For the rest of the paper, $R$ is a commutative ring with $1$. 

We define a category $\mathbf{Sh}$ of (pre-)sheaves on small categories. An object $(\mathbf{C},F)$ of $\mathbf{Sh}$ consists of a small category $\mathbf{C}$ and a sheaf $F$ on $\mathbf{C}$. A $\mathbf{Sh}$-morphism $\gamma:(\mathbf{C},F)\to(\mathbf{D},G)$ is a pair of maps $(\gamma_1,\gamma_2)$, where $\gamma_1:\mathbf{D}\to \mathbf{C}$ is a covariant functor and $\gamma_2:F\gamma_1\to G$ is a natural transformation:
\vspace{-.5em}
\[
\begin{tikzpicture}
	[scale=0.9,auto=center]
	
	\clip (-1.4,-1.4) rectangle (11, 2.7); 
	\node (1) at (0,2) {$\mathbf{C}$};
	\node (2) at (0,0)  {$\mathbf{D}$};
	\node (3) at (3,1)  {$\prescript{}{R}{Mod}$};
	\node (4) at (4.5,1) {or};
	\node (5) at (6,2) {$F(\gamma_1(x))$};
	\node (6) at (10,2) {$F(\gamma_1(y))$};
	\node (7) at (6,0) {$G(x)$};
	\node (8) at (10,0) {$G(y)$};
	\node (9) at (1.5,0.5) {};
	\node (10) at (8,1) {$\circlearrowright$};
	\draw[->]
	(1) edge node[above] {\(F\)} (3)
	(2) edge node[left] {\(\gamma_1\)} (1)
	(2) edge node[below] {\(G\)} (3)
	(5) edge node[auto] {\(\gamma_2{}_x\)} (7)
	(7) edge node[auto] {\(\scriptstyle G(y\to x)\)} (8)
	(5) edge node[auto] {\(\scriptstyle F(\gamma_1(y)\to \gamma_1(x))\)} (6)
	(6) edge node[auto] {\(\gamma_2{}_y\)} (8)
	(1) edge[double] node[below left] {\(\gamma_2\)} (9);
\end{tikzpicture}\]
\vspace{-15mm}

The composition of two morphisms $\gamma:(\mathbf{C},F)\to(\mathbf{D},G)$ and $\delta:(\mathbf{D},G)\to(\mathbf{E},H)$ is then $(\gamma_1\delta_1,\delta_2\gamma_2):(\mathbf{C},F)\to(\mathbf{E},H)$.

\begin{definition}
	\normalfont Let $\mathbf{B}$ be a small category. A \emph{bundle of sheaves} with base $\mathbf{B}$ is a contravariant functor $\xi:\mathbf{B}\to \mathbf{Sh}$.
\end{definition}

\begin{example}\label{constant}
	\begin{enumerate}
		\item A \emph{constant bundle} $\xi=\mathbf{B}\times (\mathbf{C},F)$ is a bundle of sheaves with $\xi(x)=(\mathbf{C},F)$ for all $x\in\mathbf{B}$ and $\xi(x\to y)=\id_{(\mathbf{C},F)}$ for all arrows $x\to y$.
		\item A bundle of coloured posets with base $\mathbf{B}$ in the language of \cite{ET12} is a covariant functor $\zeta$ from a poset $\mathbf{B}$ with a unique maximum to the category $\mathbf{CP}_R$ of coloured posets. Such a bundle of coloured posets gives rise to a bundle of sheaves $\xi:\mathbf{B}^{op}\to\mathbf{Sh}$, where if $\zeta(x)=(\mathbf{P},F)$, then $\xi(x)=(\mathbf{P}^{op},F)$.
		\item If $\mathbf{P}$ and $\mathbf{Q}$ are posets, then an object $F\in\mathbf{Sh}(\mathbf{P}\times\mathbf{Q})$ determines a bundle of sheaves $\xi:\mathbf{P}\to\mathbf{Sh}$. For any $x\in\mathbf{P}$, denote by $F_x$ the functor from the full subcategory $\{x\}\times\mathbf{Q}$ of $\mathbf{P}\times\mathbf{Q}$ that agrees with $F$. Then $\xi(x)=(\mathbf{Q},F_x)$ for all $x\in\mathbf{P}$ and $\xi(x\to y)=(\id_{\mathbf{Q}},F_{x\to y})$, where $F_{x\to y}|_z:F_y(y,z)\to F_x(x,z)$ agrees with $F$.
		\item We can also model a group action on a sheaf $(\mathbf{C},F)$. Let the category $\mathbf{C}_G$ have one object $\bullet$ and let the morphisms of $\mathbf{C}_G$ be given by $G$, with composition given by the group operation. Then a bundle of sheaves $\xi:\mathbf{C}_G\to\mathbf{Sh}$ with $\xi(\bullet)=(\mathbf{C},F)$ describes the action of $G$ on $(\mathbf{C},F)$.
	\end{enumerate}
\end{example}

For clarity, if $\xi$ is a bundle of sheaves with base $\mathbf{B}$ and $x\in \mathbf{B}$, then we will use the notation $\mathbf{E}_x$ for the small category that is the first coordinate of $\xi(x)$ and $F_x$ for the second coordinate of $\xi(x)$. Also, if $y\in \mathbf{E}_x$, then $\pi(y)=x$, i.e.~$\pi$ indicates which fiber $y$ comes from. Finally, for the first coordinate of the $\mathbf{Sh}$-morphism  $\xi(x_1\to x_2):(\mathbf{E}_{x_2},F_{x_2})\to(\mathbf{E}_{x_1},F_{x_1})$ we write $\xi_1(x_1\to x_2):\mathbf{E}_{x_1}\to\mathbf{E}_{x_2}$ instead of $\xi(x_1\to x_2)_1$; similarly $\xi_2(x_1\to x_2):F_{x_2}\xi_1(x_1\to x_2)\to F_{x_1}$ instead of $\xi(x_1\to x_2)_2$ for the second.

\begin{definition}\label{TotalSheaf}
	\normalfont Let $\mathbf{B}$ be a small category and $\xi$ a bundle of sheaves with base $\mathbf{B}$. The associated \emph{total sheaf} $(\mathbf{E}_\xi,F_\xi)$ consists of a small category $\mathbf{E}_\xi$ and a sheaf \hbox{$F_\xi:\mathbf{E}_\xi\to \prescript{}{R}{Mod}$}, defined as follows (also see Figure \ref{gluedup}):
	\begin{itemize}[leftmargin=*]
		\item As a small category, $\Obj(\mathbf{E}_\xi)=\bigsqcup_{x\in \mathbf{B}}\Obj(\mathbf{E}_x)$. The simple arrows of $\mathbf{E}_\xi$ are of two types. There is an arrow $y_1\to y_2$ in $\mathbf{E}_\xi$ if
		\begin{enumerate}
			\item $y_1,y_2\in \mathbf{E}_x$ for some $x\in \mathbf{B}$ and $y_1\to y_2$ is an arrow in $\mathbf{E}_x$;
			\item $x_1\to x_2$ is a non-identity arrow in $\mathbf{B}$, $y_1$ and $y_2$ are objects of $\mathbf{E}_{x_1}$ and $\mathbf{E}_{x_2}$, respectively, and we have $\xi_1(x_1\to x_2)(y_1)=y_2$.
		\end{enumerate}
		The set of all arrows of $\mathbf{E}_\xi$ is the smallest set containing the simple arrows that is closed under composition, where
		\begin{itemize}
			\item for any $x\in\mathbf{B}$, composition of arrows of type a) from $\mathbf{E}_x$ is given by the composition in $\mathbf{E}_x$,
			\item composition of arrows of type b) (and identity arrows) is given by composition in $\mathbf{B}$.
		\end{itemize}
		Additionally, we impose the commutativity of squares: if $x_1\to x_2$ is an arrow in $\mathbf{B}$ and $y_1\to y_2$ is an arrow in $\mathbf{E}_{x_1}$, then the square below commutes in $\mathbf{E}_\xi$:
		\[
		\begin{tikzpicture}
			[scale=0.8,auto=center]
			\node (1) at (0,0) {$x_1$};
			\node (2) at (4,0)  {$x_2$};
			\node (3) at (0,1)  {$y_1$};
			\node (4) at (4,1) {$\xi_1(x_1\to x_2)(y_1)$};
			\node (5) at (0,3) {$y_2$};
			\node (6) at (4,3) {$\xi_1(x_1 \to x_2)(y_2)$};
            \node (7) at (-2,0) {$\mathbf{B}:$};
            \node (8) at (-2,2) {$\mathbf{E}_\xi:$};

			\draw[->]
			(1) edge (2)
			(3) edge (4)
			(5) edge (6)
			(3) edge (5)
			(4) edge (6);
			
		\end{tikzpicture}\]
		
		\item As a sheaf, $F_\xi$ sends an object $y\in \mathbf{E}_\xi$ with $\pi(y)=x$ to $F_x(y)$. Arrows $y_1\to y_2$ of type a) from some $\mathbf{E}_x$ are sent to the map $F_x(y_1\to y_2)$; arrows $y_1\to y_2$ of type b) with $\pi(y_1)=x_1, \pi(y_2)=x_2$ are sent to $\xi_2(x_1\to x_2)_{y_1}$. Composition arrows are sent to the appropriate composition of the above maps.
	\end{itemize}
\end{definition}

\begin{figure}[h]
	\[\begin{tikzpicture}
		
		\node(1) at (-1,.05){$\mathbf{B}:$};
		\node(2) at (0,0){$x_1$};
		\node(3) at (4,0){$x_2$};
		\node(4) at (-1.5,1.55){$\mathbf{Sh}:$};
		\node(5) at (0,1.5){$(\mathbf{C}_1,F_1)$};
		\node(6) at (4,1.5){$(\mathbf{C}_2,F_2)$};
		\node(7) at (-2.8,3.5){$\mathbf{E}_\xi:$};
		\node(8) at (0,2.7){$y_1$};
		\node(9) at (4,2.7){$z_1$};
		\node(10) at (0,4.3){$y_2$};
		\node(11) at (4,4.3){$z_2$};
		\node(12) at (-2.9,6.7){$\prescript{}{R}{Mod}:$};
		\node(13) at (0,5.9){$F_1(y_1)$};
		\node(14) at (4,5.9){$F_2(z_1)$};
		\node(15) at (0,7.5){$F_1(y_2)$};
		\node(16) at (4,7.5){$F_2(z_2)$};

		\draw[->,tips=proper]
		(2) edge (3)
		(6) edge node[above] {$\scriptstyle \xi(x_1\to x_2)=\gamma$}(5)
		(8) edge[red] node[above] {$\scriptstyle \gamma_1(y_1)=z_1$} (9)
		(8) edge[blue] node[left] {$\scriptstyle  y_1\to y_2\in\mathbf{C}_1$}(10)
		(9) edge[blue] node[right] {$\scriptstyle  z_1\to z_2\in\mathbf{C}_2$}(11)
		(10) edge[red] node[below] {$\scriptstyle  \gamma_1(y_2)=z_2$} (11)
		(8) edge[purple] (11)
		(14) edge[red] node[above] {$\scriptstyle  \gamma_2|_{y_1}$}(13)
		(15) edge[blue] node[left] {$\scriptstyle  F_1(y_1\to y_2)$}(13)
		(16) edge[purple] (13)
		(16) edge[red] node[below] {$\scriptstyle \gamma_2|_{y_2}$}(15)
		(16) edge[blue] node[right] {$\scriptstyle  F_2(z_1\to z_2)$}(14)
		(2,0.3) edge node[right] {$\xi$} (2,1.2)
		(2,4.6) edge node[right] {$F_\xi$} (2,5.6);
		
	\end{tikzpicture}\]
	\caption{Constructing the total sheaf $(\mathbf{E}_\xi,F_\xi)$. Arrows of type a) are in blue, arrows of type b) are in red, and composition arrows are in purple.}
	\label{gluedup}
\end{figure}

\begin{proposition}\label{Composition arrows} Any composition arrow $f$ in $\mathbf{E}_\xi$ is equal to $gh$, for some type a) arrow $g$ and some type b) arrow $h$.
\end{proposition}
Since compositions of arrows of type a) or b) are still arrows of the same type, a composition arrow in $E$ is an alternating sequence of arrows of type a) and b). If $f$  starts with a type a) arrow and ends with a type b), then the reader can convince themselves of the conclusion of the proposition by considering the commutativity of squares in Definition \ref{TotalSheaf} and the diagram in Figure \ref{CompositionArrow}.

\begin{proposition} The pair $(\mathbf{E}_\xi,F_\xi)$ above is an object of $\mathbf{Sh}$.
\end{proposition}

It is easily checked that $\mathbf{E}_\xi$ is a small category. Furthermore, since the action of $F_\xi$ on composition arrows is defined as the composition of actions on simple arrows, functoriality of $F_\xi$ follows from the functoriality of $\xi$ and $F_x$ for all $x\in \mathbf{B}$.

\begin{figure}
\begin{tikzpicture}[scale=1.3,auto=center,tips=proper]
    every edge/.style = {draw, -Latex}
    
  \node (1) at (0,0) {};
  \node (2) at (0,0.5)  {};
  \node (3) at (0.5,0.5)  {};
  \node (4) at (0.5,1) {};
  \node (5) at (1,1) {};
  \node (7) at (1.5,1.5) {};
  \node (8) at (1.5,2) {};
  \node (9) at (2,2) {};
  \node (10) at (-1,1) {a)};
  \node (11) at (1,-0.7) {b)};
  \node (12) at (-0.25,-0.25) {};
  \node (13) at (-0.25,2.25) {};
  \node (14) at (2.25,-0.25) {};
  \node (15) at (0.5,0) {};
  \node (16) at (1,0) {};
  \node (17) at (1.5,0) {};
  \node (18) at (2,0) {};
  \node (19) at (1,0.5) {};
  \node (20) at (1.5,0.5) {};
  \node (21) at (2,0.5) {};
  \node (22) at (1.5,1) {};
  \node (23) at (2,1) {};
  \node (24) at (2,1.5) {};

 \draw[-]
  (12.center) edge (13.center)
  (12.center) edge (14.center);
  
 \draw[<-,tips=proper,line width=0.3mm]
  (2.center) edge (1.center)
  (3.center) edge (2.center)
  (4.center) edge (3.center)
  (5.center) edge (4.center)
  (8.center) edge (7.center)
  (9.center) edge (8.center);
  
 \draw[<-,tips=proper]
  (24.center) edge (7.center)
  (23.center) edge (22.center)
  (21.center) edge (20.center)
  (19.center) edge (3.center)
  (22.center) edge (20.center)
  (20.center) edge (17.center)
  (5.center) edge (19.center)
  (19.center) edge (16.center)
  (3.center) edge (15.center);
  
 \draw[<-,tips=proper,dashed]
  (9.center) edge (18.center)
  (18.center) edge (1.center);
  
 \draw[dotted]
  (5) edge (7)
  (22) edge (7)
  (23) edge (24)
  (5) edge (22)
  (19) edge (20)
  (16) edge (17);

\end{tikzpicture}
\caption{An alternating sequence of a) and b) arrows and the resulting commutative grid.}
\label{CompositionArrow}
\end{figure}

\section{Cochain modules of $(\mathbf{C},F)$}\label{sec:cochain}

A sheaf $F$ on $\mathbf{C}$ acts as a covariant functor on the poset of simplices of the nerve $N\mathbf{C}$ of $\mathbf{C}$. For $\sigma=x_0\to x_1\to\cdots\to x_i$, $\tau=x_{i_0}\to\cdots\to x_{i_k}$, we set $F(\sigma)=F(x_0)$, and
\[
F(\tau\subseteq \sigma)=F(x_0\to x_{i_0})=F(x_{i_0})\to F(x_0),
\]
where the arrow $x_{0}\to x_{i_0}$ is given by the appropriate composition of arrows in $\sigma$. The module for the $k$-cochains $(k\geq0)$ is
\[
\mathcal{S}^k(N\mathbf{C};F)=\prod_{\sigma} F(\sigma),
\]
where the product ranges over all $k$-simplices $\sigma=x_0\to\cdots\to x_k$. For $k<0$, $\mathcal{S}^k(N\mathbf{C};F)=0$. The differential $d^k:\mathcal{S}^{k-1}(N\mathbf{C};F)\to \mathcal{S}^{k}(N\mathbf{C};F)$  for $k>0$ is given by
\[
(d^k u)|_{\sigma}=\sum_{j=0}^k (-1)^j F(\sigma_j\subseteq \sigma)(u|_{\sigma_j}),
\]
where $\sigma=x_0\to x_1\to\cdots\to x_k$, $u\in\mathcal{S}^{k-1}(N\mathbf{C};F)$, and $\sigma_j=x_0\to\cdots\to \hat{x}_{j}\to\cdots\to x_k$. For $k\leq 0$, $d^k=0$. It is easily seen that $\mathcal{S}^\bullet(N\mathbf{C};F)$ is a chain complex. For notational brevity, we will also write it as $\mathcal{S}^\bullet(\mathbf{C};F)$.

Given a $\mathbf{Sh}$-morphism $\gamma:(\mathbf{C},F)\to (\mathbf{D},G)$, there is an induced map $\gamma^\bullet:\mathcal{S}^\bullet(\mathbf{C};F)\to\mathcal{S}^\bullet(\mathbf{D};G)$ defined by
\[
\gamma^\bullet u|_\sigma=\gamma_{2 x_0}(u|_{\gamma_1(\sigma)}).
\]
\begin{lemma}\label{chainmap}
	The induced map $\gamma^\bullet$ is a well-defined chain map.
\end{lemma}
This follows from an easy calculation, using the naturality of $\gamma_2$.

We have thus defined a covariant functor $\mathcal{S}^\bullet:\mathbf{Sh}\to\mathbf{Ch}_R$, from pairs of small categories and sheaves to chain complexes over $R$. In particular, for each $q\in\bZ$ we have a covariant functor $\mathcal{S}^q:\mathbf{Sh}\to\prescript{}{R}{Mod}$. Since homology is a functor from chain complexes to graded $R$-modules, we also have a covariant functor $H^\bullet\mathcal{S}^\bullet:\mathbf{Sh}\to\mathbf{Gr}\prescript{}{R}{Mod}$. In particular, for each $q\in\bZ$ we have a covariant functor $H^q\mathcal{S}^\bullet:\mathbf{Sh}\to\prescript{}{R}{Mod}$.

Given a bundle $\xi:\mathbf{B}\to \mathbf{Sh}$, the above gives us two sheaves on $\mathbf{B}$. For any $q\in\bZ$ the \emph{$q$-cochain sheaf} of $\mathbf{B}$ is the sheaf $\mathcal S^q:\mathbf{B}\to \prescript{}{R}{Mod}$, i.e.~the composition
\[
\mathbf{B}\overset{\xi}{\longrightarrow}\mathbf{Sh}\overset{ \mathcal S^q}{\longrightarrow}\prescript{}{R}Mod.
\]
Similarly, the \emph{homology of the fibres} sheaf of $\mathbf{B}$ is the sheaf $\mathcal H_{fib}^q:\mathbf{B}\to \prescript{}{R}{Mod}$, i.e.~the composition
\[
\mathbf{B}\overset{\xi}{\longrightarrow}\mathbf{Sh}\overset{\mathcal S^\bullet}{\longrightarrow}\mathbf{Ch}_R\overset{\mathcal H^q}{\longrightarrow}\prescript{}{R}Mod.
\]
Explicitly, if $x\in \mathbf{B}$, then $\mathcal H_{fib}^q(x)=H^q(\mathbf{E}_x;F_x)$.

\section{The bicomplex $\mathcal{S}^{p}(\mathbf{B};\mathcal{S}^{q})$}\label{sec:bicomplex}

We want to construct a bicomplex by taking the $p$-cochain sheaf of  $(\mathbf{B},\mathcal{S}^q)$.

Let $\xi:\mathbf{B}\to\mathbf{Sh}$ be a bundle of sheaves and suppose $x\to y$ is an arrow in $\mathbf{B}$. We have the commutative square
\[
\begin{tikzpicture}
	
	[scale=1.4,auto=center,tips=proper]
	every edge/.style = {draw, -Latex}
	
	\node (1) at (0,0) {$\mathcal{S}^{q-1}(\mathbf{E}_y;F_y)$};
	\node (2) at (0,1.5)  {$\mathcal{S}^{q-1}(\mathbf{E}_x;F_x)$};
	\node (3) at (3.5,1.5)  {$\mathcal{S}^{q}(\mathbf{E}_x;F_x)$};
	\node (4) at (3.5,0) {$\mathcal{S}^{q}(\mathbf{E}_y;F_y)$};
	
	\draw[->,tips=proper]
	(1) edge (2)
	(2) edge node[above] {$d$} (3);
	
	\draw[->,tips=proper]
	(4) edge (3)
	(1) edge node[above] {$d$} (4);
\end{tikzpicture}\]
\noindent where the vertical maps are the chain map from Lemma \ref{chainmap} induced by $\xi(x\to y)$. In particular, the differential $d$ induces a $\mathbf{Sh}$-morphism  $\gamma:(\mathbf{B},\mathcal{S}^{q-1})\to(\mathbf{B},\mathcal{S}^q)$, where $\gamma_1$ is the identity functor and $\gamma_2$ are the differentials at each object of $\mathbf{B}$. This gives us the induced map 
\[
\gamma^\bullet:\mathcal{S}^\bullet(\mathbf{B};\mathcal{S}^{q-1})\to\mathcal{S}^\bullet(\mathbf{B};\mathcal{S}^q).
\]

Applying this for all $q\in\bZ$ gives a grid of commutative squares of the form:
\[
\begin{tikzpicture}[scale=1,auto=center,tips=proper]
	\node (1) at (0,0) {$\mathcal{S}^{p-1}(\mathbf{B};\mathcal{S}^{q-1})$};
	\node (2) at (0,1.2)  {$\mathcal{S}^{p-1}(\mathbf{B};\mathcal{S}^q)$};
	\node (3) at (3,1.2)  {$\mathcal{S}^{p}(\mathbf{B};\mathcal{S}^{q})$};
	\node (4) at (3,0) {$\mathcal{S}^{p}(\mathbf{B};\mathcal{S}^{q-1})$};
	
	\draw[->,tips=proper]
	(1) edge (2)
	(2) edge (3);
	
	\draw[->,tips=proper]
	(4) edge (3)
	(1) edge (4);
	
\end{tikzpicture}\]
To make the squares anti-commute instead, we apply the usual `Jedi sign trick', i.e.~we include a factor of $-1$ in every other horizontal map. We will be concerned with this bicomplex in particular in later chapters, so we will sometimes refer to it as just $\mathcal{K}_\xi^{p,q}$. Explicitly, we have
\[
\mathcal{K}_\xi^{p,q}=\mathcal{S}^{p}(\mathbf{B};\mathcal{S}^{q});
\]
if we denote 
\[
\sigma=x_0\to \ldots \to x_p\in N\mathbf{B} \text{ and }\tau=y_0\to\ldots\to y_q\in N\mathbf{E}_{x_0},
\]
then the vertical differential  $d^v:\mathcal{S}^p(\mathbf{B};\mathcal{S}^{q-1})\to\mathcal{S}^p(\mathbf{B};\mathcal{S}^{q})$ is defined by
\[
(d^vu)|_{\sigma,\tau}=F_{x_0}(y_0\to y_1)(u|_{\sigma,\tau_0})+\sum_{j=1}^q(-1)^j(u|_{\sigma,\tau_j})
\]
and the horizontal differential  $d^h:\mathcal{S}^{p-1}(\mathbf{B};\mathcal{S}^{q})\to\mathcal{S}^p(\mathbf{B};\mathcal{S}^{q})$ is defined by
\[
(d^hu)|_{\sigma,\tau}=(-1)^{p+q}\left(\gamma_{y_0}(u|_{\sigma_0,\gamma_1(\tau)})+\sum_{i=1}^p(-1)^i(u|_{\sigma_j,\tau})\right),
\]
where $\xi_2(x_0\to x_1)=\gamma$.

We can place the modules $\mathcal{K}_\xi^{p,q}$ on the $E_0$ page of a spectral sequence and use the vertical maps as the differentials on that page. We can further use the quotients of the horizontal maps for the differentials on the $E_1$ page.

\begin{proposition}
	The $E_2$ page of the spectral sequence defined above has
	\[
	E_2^{p,q}=H^p(\mathbf{B};\mathcal{H}^q_{fib}(\xi)).
	\]
\end{proposition}
\begin{proof}
	Note that the differentials on the $E_2$ page are of degree $(2,-1)$. Consider the following diagram
	\[\begin{tikzpicture}[scale=1,auto=center,tips=proper]
		\node (1) at (1,1.5) {$\xi$};
		\node (2) at (3,1.5) {$(\mathbf{B},\mathcal{S}^\bullet)$};
		\node (3) at (7,1.5) {$\mathcal{S}^p(\mathbf{B};\mathcal{S}^\bullet)$};
		\node (4) at (11,1.5) {$H^\bullet(\mathcal{S}^p(\mathbf{B};\mathcal{S}^\bullet))$};
		\node (5) at (7,0) {$(\mathbf{B},\mathcal{H}^\bullet_{fib}(\xi))$};
		\node (6) at (11,0) {$\mathcal{S}^p(\mathbf{B};\mathcal{H}^\bullet_{fib}(\xi))$};
		
		\draw[->,tips=proper]
		(1) edge node[above] {$\mathcal{S}^\bullet$} (2)
		(2) edge node[above] {$\mathcal{S}^p$} (3)
		(3) edge node[above] {$H^\bullet$} (4)
		(2.east) edge[bend right] node[auto] {$H^\bullet$} (5.west)
		(5) edge node[auto] {$\mathcal{S}^p$} (6);

	\end{tikzpicture}\]
	
	The top path is how we get the modules in a given column on the $E_1$ page -- we take vertical homology of a column in $E_0$. On the other hand, taking horizontal homology of rows formed by $\mathcal{S}^p(\mathbf{B};\mathcal{H}^q_{fib}(\xi))$ clearly gives the required modules $H^p(\mathbf{B};\mathcal{H}^q_{fib}(\xi))$. It is then enough to show that the two graded modules at the ends of the two paths are equal for each $p\in\bZ$. This follows directly from cohomology commuting with the direct product.
\end{proof}

Now, there is a total complex associated to $\mathcal{K}_\xi^{p,q}$. We will denote it as $T_\xi^\bullet$. Explicitly,
\[
T^n_\xi:=\prod_{p+q=n}\mathcal{K}_\xi^{p,q},
\]
with $d=d^h+d^v$. Then, from the general construction of a spectral sequence from a bicomplex (see \cite{weibel}) and from the above proposition, we have the sheaf cohomological version of \cite[Proposition 2.2]{ET12}:

\begin{proposition}\label{Spectral} If $\xi:\mathbf{B}\to\mathbf{Sh}$ is a bundle of sheaves, then there is a spectral sequence
	\[
	E_2^{p,q}=H^p(B;\mathcal{H}_{fib}^q)\Longrightarrow H^\bullet(T^\bullet_\xi).
	\]
\end{proposition}

\section{Grid traversals}\label{sec:traversals}

For a given poset bundle of sheaves $\xi$, we define a chain map \hbox{$\omega:\mathcal{S}^\bullet(\mathbf{E}_\xi;F_\xi)\to T^\bullet_\xi$}, where $\mathcal{S}^\bullet(\mathbf{E}_\xi;F_\xi)$ is the chain complex constructed in \S \ref{sec:cochain} on the total sheaf of $\xi$ (recall Definition \ref{TotalSheaf}), and $T^\bullet_\xi$ is the total complex associated to the bicomplex $\mathcal{K}_\xi^{\bullet,\bullet}$ constructed in \S \ref{sec:bicomplex}. To do that, if $\sigma=x_0\to\cdots\to x_p\in N\mathbf{B}$ and $\tau=y_0\to\cdots\to y_q\in N\mathbf{E}_{x_0}$, then to each pair $(\sigma,\tau)$ we will associate a (signed) combination of all traversals of a particular grid in $\mathbf{E}_\xi$.

To form this grid, we lay out $\sigma$ and $\tau$ (see Figure \ref{GridsTraversals}) and complete the grid using the morphisms $\xi(x_i\to x_{i+1})$ -- on the figure we denote $y_{0,j}=y_j$ and $y_{i+1,j}=\xi_1(x_i\to x_{i+1})(y_{i,j})$.
\begin{figure}
\[
\begin{tikzpicture}[xscale=1.3,auto=center,tips=proper]
	
	\node (1) at (0,0) {$\sigma$};
	\node (2) at (1,0)  {$x_0$};
	\node (3) at (2,0)  {$x_1$};
	\node (4) at (3,0) {$\cdots$};
	\node (5) at (4,0) {$x_p$};
	\node (7) at (1,1) {$y_{0,0}$};
	\node (8) at (1,2) {$\vdots$};
	\node (9) at (1,3) {$y_{0,q}$};
	\node (10) at (1,4) {$\tau$};
	\node (11) at (1,3.5) {$\shortparallel$};
	\node (12) at (0.5,0) {$=$};
	\node (13) at (2,1) {$y_{1,0}$};
	\node (14) at (2,2) {$\vdots$};
	\node (15) at (2,3) {$y_{1,q}$};
	\node (16) at (3,1) {$\cdots$};
	\node (17) at (3,3) {$\cdots$};
	\node (18) at (4,1) {$y_{p,0}$};
	\node (19) at (4,2) {$\vdots$};
	\node (20) at (4,3) {$y_{p,q}$};

	\draw[->,tips=proper]
	(2) edge (3)
	(3) edge (4)
	(4) edge (5)
	(7) edge (8)
	(8) edge (9)
	(7) edge (13)
	(13) edge (14)
	(14) edge (15)
	(9) edge (15)
	(13) edge (16)
	(15) edge (17)
	(16) edge (18)
	(17) edge (20)
	(18) edge (19)
	(19) edge (20);
	
\end{tikzpicture}\qquad\qquad
\begin{tikzpicture}[scale=1.3,auto=center,tips=proper]
	
	\node (1) at (0,0) {};
	\node (2) at (1,0)  {};
	\node (3) at (2,0)  {};
	\node (4) at (0,1) {};
	\node (5) at (1,1) {};
	\node (6) at (2,1) {};
	\node (7) at (3,0) {};
	\node (8) at (3,1)  {};
	\node (9) at (3,2)  {};
	\node (10) at (0,2) {};
	\node (11) at (1,2) {};
	\node (12) at (2,2) {};  
	
	\draw[->,tips=proper]
	
	(2) edge (3)
	(4) edge (5)
	(5) edge (6)
	(1) edge (4)
	(3) edge (6)
	(3) edge (7)
	(10) edge (11)
	(4) edge (10)
	(8) edge (9)
	(6) edge (12)
	(6) edge (8)
	(7) edge (8);
	
	\draw [->, tips=proper, line width=.35mm]
	(1) edge (2)
	(2) edge (5)
	(5) edge (11)
	(11) edge (12)
	(12) edge (9);

\end{tikzpicture}
\]
\caption{The grid of $(\sigma,\tau)$ (left) and an example grid traversal (right).}
\label{GridsTraversals}
\end{figure}

A \emph{grid traversal} $z\in N\mathbf{E}_\xi$ of the grid of $(\sigma,\tau)$ is a chain of length $(p+q)$ of arrows in the grid. In particular, each arrow in $z$ is either 
	\[
	\xi_1(x_0\to x_i)(y_j\to y_{j+1})  \text{ or } y_{i,j}\to \xi_1(x_i\to x_{i+1})(y_{i,j}).
	\]
	Note that these correspond to type a) and type b) in Definition \ref{TotalSheaf}.

For each grid traversal $z$ of the grid of $(\sigma,\tau)$, define
	\[
	m(z)=\#\{\text{squares in the grid below and to the right of }z\}.
	\]
	Furthermore, define $\varsigma(q)=\left\lceil{\dfrac{q}{2}}\right\rceil=\min\left\{n\in\bZ\mid n\geq \dfrac{q}{2}\right\}$. 
	
	We can now define the chain map we are interested in. The map $\omega:\mathcal S^\bullet(\mathbf{E}_\xi;F_\xi)\to T^\bullet_\xi$ is defined, for any $u\in \mathcal S^\bullet(\mathbf{E}_\xi;F_\xi)$, by
	\[
	(\omega u)|_{\sigma,\tau}=(-1)^{\varsigma(q)}\sum_z (-1)^{m(z)}u|_z,
	\]
	where the sum is taken over all traversals $z$ of the grid of $(\sigma,\tau)$.

\begin{proposition}
	The map $\omega$ defined above is a chain map.
\end{proposition}
\begin{proof}
    The argument here is analogous to the argument showing that a similar map is a chain map in \cite[Proposition 5.2]{ET12}.
\end{proof}

\section{Technical Tools}\label{sec:technical}

Up to this point, for a bundle of sheaves $\xi:\mathbf{B}\to\mathbf{Sh}$, we have constructed the total sheaf $(\mathbf{E}_\xi,F_\xi)$ and its simplicial complex $\mathcal{S}^\bullet(\mathbf{E}_\xi;F_\xi)$, as well as the bicomplex $\mathcal{K}_\xi^{\bullet,\bullet}$ and its total complex $T_\xi^\bullet$. We know that the spectral sequence of the bicomplex converges to $H^\bullet T^\bullet_\xi$, but we would like to identify cases where it converges to the cohomology of the total sheaf. In this section, we show that, under certain (fairly strong) assumptions, the chain map $\omega$ is a quasi-isomorphism. This puts the results of \cite{ET12} into a sheaf cohomology setting.

\begin{definition}
	\normalfont A bundle of sheaves $\xi:\mathbf{B}\to\mathbf{Sh}$ is a \emph{poset bundle of sheaves} if both $\mathbf{B}$ and $\mathbf{E}_x$ for all $x\in\mathbf{B}$ are finite posets.
\end{definition}

From this point on, all small categories in sight are assumed to be finite posets. If $x,y\in \mathbf{B}$, we say that \emph{$y$ covers $x$} (denoted $x\prec y$) if, whenever $z\in \mathbf{B}$ is such that $x\leq z\leq y$, we have $z=x$ or $z=y$. We also say that $\mathbf{B}$ \emph{has a $0$} (or is a poset with $0$) if $\mathbf{B}$ has a unique minimal element $0\in \mathbf{B}$. 

Now, for an element $x\in\mathbf{B}$, define $\mathbf{B}_{\geq x}$ and $\mathbf{B}_{\not\geq x}$ to be the full subcategories of $\mathbf{B}$ with
\[
\Obj\mathbf{B}_{\geq x}:=\{z\in \Obj\mathbf{B}\mid x\leq z\}\, \text{ and }\, \Obj\mathbf{B}_{\not\geq x}:=\Obj\mathbf{B}\backslash \Obj\mathbf{B}_{\geq x}.
\]
Note that both $\mathbf{B}_{\geq x}$ and $\mathbf{B}_{\not\geq x}$ inherit the poset structure of $\mathbf{B}$. We will occasionally omit $\Obj$ when we refer to the objects of a poset category if the meaning is clear from context.

The key property we will exploit is the following.

\begin{definition}\label{admissible}
	\normalfont Assume $\mathbf{B}$ is a poset.
	\begin{enumerate}
		\item Let $\mathbf{B}_1$ and $\mathbf{B}_2$ be full subposets of $\mathbf{B}$. We call $\mathbf{B}$ \emph{admissible for} $\mathbf{B}_1,\mathbf{B}_2$ if
		\begin{itemize}
			\item $\mathbf{B}_1\cap \mathbf{B}_2=\emptyset$,
			\item $\mathbf{B}_1\cup\mathbf{B}_2=\mathbf{B}$,
			\item there are no $x\in\mathbf{B}_2$ and $y\in \mathbf{B}_1$ with $x\leq y$, and
			\item for all $x\in\mathbf{B}_1$, the full subposet $\{y\in\mathbf{B}_2\mid x\leq y\}\subseteq \mathbf{B}_2$ is non-empty and has a unique minimum.
		\end{itemize}
		\item We call $\mathbf{B}$ \emph{admissible for} $x\in\mathbf{B}$ if $\mathbf{B}$ is admissible for $\mathbf{B}_{\not\geq x},\mathbf{B}_{\geq x}$. Note that the first three requirements of admissibility are automatically satisfied for $\mathbf{B}_{\not\geq x},\mathbf{B}_{\geq x}$ (see bottom of Figure \ref{MoreIntersection}). We also denote the poset in the last requirement by
		\[
		\mathbf{B}_{\geq x}^{\geq y}:=\{z\in \mathbf{B}_{\geq x}\mid y\leq z\}=\mathbf{B}_{\geq x}\cap\mathbf{B}_{\geq y}.
		\]
		\item We call $\mathbf{B}$ \emph{recursively admissible} if $\mathbf{B}$ has a $0$ and either
		\begin{itemize}
			\item $\mathbf{B}$ is Boolean of rank $1$, or
			\item $\mathbf{B}$ is admissible for some $x\succ 0$ and both $\mathbf{B}_{\geq x}$ and $\mathbf{B}_{\not\geq x}$ are recursively admissible.
		\end{itemize}
	\end{enumerate}
\end{definition}

If we have a poset bundle of sheaves $\xi:\mathbf{B}\to \mathbf{Sh}$ and a subcategory $\mathbf{C}$ of $\mathbf{B}$, we can restrict the bundle $\xi$ to $\mathbf{C}$ to obtain another bundle $\xi_\mathbf{C}:\mathbf{C}\to \mathbf{Sh}$ with total sheaf $(\mathbf{E}_{\xi_\mathbf{C}};F_{\xi_\mathbf{C}})$. When the bundle $\xi$ is clear from context, we will just use $(\mathbf{E}_{\mathbf{C}};F_{\mathbf{C}})$. Note that we use $(\mathbf{E}_x;F_x)$ for the sheaf $\xi(x)$ when $x$ is an object of $\mathbf{B}$, which (almost) coincides with $(\mathbf{E}_\mathbf{C};F_\mathbf{C})$ when $\mathbf{C}$ is the subcategory of $\mathbf{B}$ consisting only of $x$ and its identity arrow.

The next lemma shows how admissibility of $\mathbf{B}$ extends to $\mathbf{E}_\xi$. 

\begin{lemma}\label{gluedminimal}
	Let $\mathbf{B}$ be admissible for some $x\in\mathbf{B}$ and $\xi:\mathbf{B}\to \mathbf{Sh}$ be a poset bundle of sheaves with total sheaf $(\mathbf{E}_\xi;F_\xi)$. Then $\mathbf{E}_\xi$ is admissible for $\mathbf{E}_{\mathbf{B}_{\not\geq x}},\mathbf{E}_{\mathbf{B}_{\geq x}}$.

\end{lemma}
\begin{proof}
	It is immediate that $\mathbf{E}_{\mathbf{B}_{\not\geq x}}$ and $\mathbf{E}_{\mathbf{B}_{\geq x}}$ are disjoint, that $\mathbf{E}_{\mathbf{B}_{\not\geq x}}\cup\mathbf{E}_{\mathbf{B}_{\geq x}}=\mathbf{E}_\xi$, and that there is no arrow from an object of $\mathbf{E}_{\mathbf{B}_{\not\geq x}}$ to an object of $\mathbf{E}_{\mathbf{B}_{\geq x}}$. It remains to show that for all $w\in\mathbf{E}_{\mathbf{B}_{\not\geq x}}$, the subposet $\{z\in \mathbf{E}_{\mathbf{B}_{\geq x}}\mid w\leq z\}$ has a unique minimal element.
	
	Since $w\in \mathbf{E}_{\mathbf{B}_{\not\geq x}}$, $w$ is an element of a particular $\mathbf{E}_y$ for some $y\in \mathbf{B}_{\not\geq x}$. By the admissibility of $\mathbf{B}$, that means that the poset $\mathbf{B}_{\geq x}^{\geq y}$ has a unique minimum, say $v$. Then $y\leq v$ and thus there is an arrow $y\to v$ in $\mathbf{B}$. Denote the sheaf morphism given by this arrow as $\gamma$. By the construction of the total sheaf, we have that $w\leq \gamma_1(w)$. 
	
	Suppose $w\leq z$ for some $z\in \mathbf{E}_{\mathbf{B}_{\geq x}}$ and suppose $z\in \mathbf{E}_u$, $u\in\mathbf{B}_{\geq x}$. Then by our argument in Proposition \ref{Composition arrows} we have a $z_0\in\mathbf{E}_u$ with $w\leq z_0\leq z$ and an arrow $y\to u$ giving rise to a sheaf morphism $\gamma'$. Thus $u$ is in $\mathbf{B}_{\geq x}^{\geq y}$, not just $\mathbf{B}_{\geq x}$. Since $v$ is the minimal element of $\mathbf{B}_{\geq x}^{\geq y}$, we have that $v\leq u$. But there is a unique arrow $y\to u$, so $\gamma'_1$ factors through $\mathbf{E}_v$ and the sheaf morphism given by $v\to u$ maps $\gamma_1(w)$ to $z_0$. This means that $\gamma_1(w)\leq z_0\leq z$, therefore $\gamma_1(w)$ is the minimum of the set $\{z\in \mathbf{E}_{\mathbf{B}_{\geq x}}\mid w\leq z\}$. Refer to Figure \ref{MoreIntersection} for the relevant objects.\qedhere
	\begin{figure}[h]
		\[\begin{tikzpicture}[scale=1]
			
			\node (1) at (0,0) {\includegraphics[scale=.06]{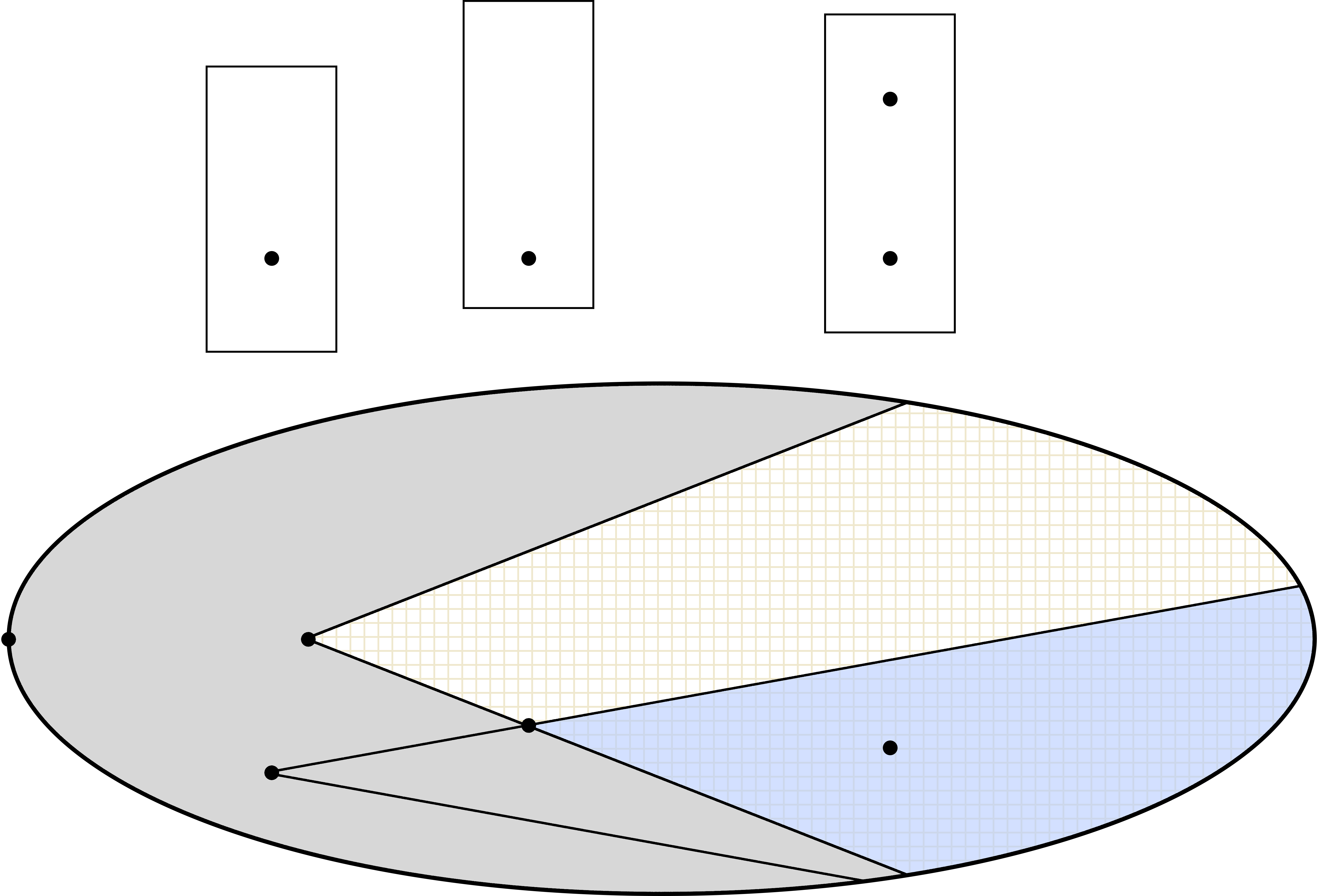}};
			\node(2) at (-57/10,-15.9/10){$0$};
			\node(3) at (-29.2/10,-12/10){$x$};
			\node(4) at (-35/10,-27.2/10){$y$};
			\node(5) at (-13/10,-20/10){$v$};
			\node(6) at (22/10,-25/10){$u$};
			\node(7) at (-32.3/10,13/10){$\scriptstyle w$};
			\node(8) at (-32.3/10,35/10){$\mathbf{E}_y$};
			\node(9) at (-10.5/10,19/10){$\scriptstyle\gamma_1(w)$};
			\node(10) at (19.2/10,13/10){$\scriptstyle z_0$};
			\node(11) at (19.2/10,32/10){$\scriptstyle z$};
			\node(12) at (-10.5/10,40/10){$\mathbf{E}_v$};
			\node(13) at (19.2/10,39/10){$\mathbf{E}_u$};
			
			\node(14) at (-31/10,15/10){};
			\node(15) at (18/10,15/10){};

            \node(16) at (4,.3) {$\mathbf{B}$};
            \node(17) at (-4.2,-1) {$\mathbf{B}_{\not\geq x}$};
            \node(18) at (2.7,-.4) {$\mathbf{B}_{\geq x}$};
            \node(19) at (3.4, -2.2) {$\mathbf{B}_{\geq x}^{\geq y}$};
			
			\draw[->,tips=proper]
			(-30/10,-27/10) edge[bend right=10] node[below right=-.07 and .1]{$\scriptstyle \gamma$}(-12/10,-24/10)
			(-8/10,-23.3/10) edge node[below right] {$\scriptstyle \gamma'$}(17.5/10,-24.7/10)
			(-30.5/10,15.9/10) edge node[above] {$\scriptstyle \gamma_1$}(-12.5/10,15.9/10)
			(-8.8/10,15.9/10) edge (17.5/10,15.9/10)
			(19.3/10,17.5/10) edge (19.3/10,27.5/10);
			\draw[->,tips=proper](14) edge[out=350,in=190] node[below right=0 and .5] {$\scriptstyle \gamma'_1$}(15);
			
			\draw[dashed](-32.15/10,-27.2/10) edge (-32.15/10,8.6/10)
			(-10.75/10,-23/10) edge (-10.75/10,12/10)
			(19.3/10,-25/10) edge (19.3/10,9.5/10);
		\end{tikzpicture}\]
		\caption{The poset $\mathbf{B}$ with $y\in\mathbf{B}_{\not\geq x}$ and the fibers over $y$, $v$ and $u$.}
		\label{MoreIntersection}
	\end{figure}
\end{proof}

For constant bundles over certain posets we have a calculation of the cohomology of the total complex. We recall the following facts about morphisms of spectral sequences (see \cite{maclane}). If $E,E'$ are spectral sequences constructed from filtrations (or bicomplexes), then a morphism of filtrations (or of bicomplexes) induces a morphism $E\to E'$. A spectral sequence $E$ is \emph{bounded below} if for each degree $n$ there is an integer $s=s(n)$ such that $E_0^{p,q}=0$ when $p<s$ and $p+q=n$. If $E,E'$ are bounded below spectral sequences and $f:E\to E'$ is a morphism, such that for some $r$ the homomorphisms $f_r^{p,q}:E_r^{p,q}\to E'^{p,q}_r$ are isomorphisms for each $p$ and $q$, then the maps $f_\infty^{p,q}:E_\infty^{p,q}\to E'^{p,q}_\infty$ are also isomorphisms. The above is known as the \emph{mapping lemma}.

\begin{proposition}\label{alpha}
	Suppose $\mathbf{B}$ is a poset, $x\in \mathbf{B}$ is a unique minimum, and $(\mathbf{C},F)$ is an object of $\mathbf{Sh}$. If $\xi=\mathbf{B}\times (\mathbf{C},F)$ is a constant bundle (recall Example \ref{constant}), then there is a chain map $\varphi^\bullet:\mathcal{S}^\bullet(\mathbf{C};F)\to T^\bullet_\xi$ such that the induced map $\varphi^\bullet:H^\bullet(\mathbf{C};F)\to H^\bullet T^\bullet_\xi$ is an isomorphism.
\end{proposition}
\begin{proof}
	It is straightforward to see why $\mathcal{S}^\bullet(\mathbf{C};F)$ is quasi-isomorphic to $T^\bullet_\xi$. The $E_2$ page of the spectral sequence for $\xi$ has
	\[
	E_2^{p,q}=H^p(\mathbf{B},\Delta H^q(\mathbf{C};F)).
	\]
	Since the right-hand side is the cohomology of a constant sheaf, the only non-zero positions on the $E_2$ page are in the column $p=0$; so the sequence collapses and we can read off $H^\bullet T^\bullet_\xi$. Explicitly,
	\[
	H^p(\mathbf{B};\Delta H^q(\mathbf{C};F))=\left\{\begin{array}{c c l}H^q(\mathbf{C};F), && \text{ if }p=0, \\ 0, && \text{ otherwise}.\end{array}\right.
	\]
	So $H^\bullet(\mathbf{C};F)\cong H^\bullet T^\bullet_\xi$. It is still useful to describe the explicit quasi-isomorphism; we will use a version of this explicit chain map in the proof of Proposition \ref{first}.
	
	First consider the constant sheaf $(\mathbf{P},\Delta A)$, where $\mathbf{P}$ is a poset with a unique minimum. Recall that
	\[
	H^n(\mathbf{P};\Delta A)\cong\left\{\begin{array}{c c l}A, && \text{ if }n=0, \\ 0, && \text{ otherwise}.\end{array}\right.
	\]
	
	Our first goal is to find an explicit map for the isomorphism above. So let \hbox{$u\in\mathcal{S}^0(\mathbf{P};\Delta A)$} be such that $du=0$. Since we have a unique minimum $0$, for any $x\in \mathbf{P}$, there is an arrow $0\leq x$ in $\mathbf{P}$. Then $0=du|_{0 \leq x}=u|_x - u|_0$, so $u|_x=u|_0$ for all $x\in \mathbf{P}$. Denote such a constant element of $\mathcal{S}^0(\mathbf{P};\Delta A)$ by $u_a$ if $u_a|_x=a\in A$ for all $x\in \mathbf{P}$. So the isomorphism we are looking for is \hbox{$\theta:A\to H^0(\mathbf{P},\Delta A):a\mapsto u_a$}.
	
	Now consider the (trivial) chain complex $\iota^\bullet(A)$ defined by
	\[
	\iota^n(A)=\left\{\begin{array}{c c l} A, &&\text{ if }n=0, \\ 0, &&\text{ otherwise},\end{array}\right.
	\]
	and $d^n_{\iota^\bullet(A)}=0$ for all $n$. Define the map $\psi^\bullet:\iota^\bullet(A)\to \mathcal{S}^\bullet(\mathbf{P};\Delta A)$ as
	\[
	\psi^n=\left\{\begin{array}{c c l}\theta, &&\text{ if }n=0,\\ 0, &&\text{ otherwise.}\end{array}\right.
	\]
	\[
	\begin{tikzpicture}[scale=2.2,auto=center,tips=proper]
		
		\node (1) at (0,.7) {$\cdots$};
		\node (2) at (1.5,.7)  {$0$};
		\node (3) at (3,.7)  {$A$};
		\node (4) at (4.5,.7) {$0$};
		\node (5) at (6,.7) {$\cdots$};
		\node (6) at (0,0) {$\cdots$};
		\node (7) at (1.5,0) {$0$};
		\node (8) at (3,0) {$\mathcal{S}^0(\mathbf{P};\Delta A)$};
		\node (9) at (4.5,0) {$\mathcal{S}^1(\mathbf{P};\Delta A)$};
		\node (10) at (6,0) {$\cdots$};
		
		\draw[->,tips=proper]
		(1) edge (2)
		(2) edge (3)
		(3) edge (4)
		(4) edge (5)
		(6) edge (7)
		(7) edge (8)
		(8) edge (9)
		(9) edge (10)
		(2) edge node[auto] {$0$} (7)
		(3) edge node[auto] {$\theta$} (8)
		(4) edge node[auto] {$0$} (9);
	\end{tikzpicture}\]
	To see this is a chain map, note that $\theta(a)\in \ker(d^0)$, so $d\theta=0$. All other squares commute since all compositions are the $0$ map.
	
	Crucially, $\psi^\bullet$ is a quasi-isomorphism. This is because $H^0\iota^\bullet(A)=A$ and by construction $\theta$ induces the isomorphism $H^0\iota^\bullet(A)\to H^0\mathcal{S}^\bullet(\mathbf{P};\Delta A)$. Note that the map $-\psi^\bullet$ is also a quasi-isomorphism, since $-\theta$ induces $-\id:A\to A$ in homology.

	Returning to the case of the constant bundle $\xi=\mathbf{B}\times (\mathbf{P},F)$, we can now define \hbox{$\varphi^n:\mathcal{S}^n(\mathbf{C};F)\to T^n_\xi$} by
	\[
	\varphi u|_{\sigma,\tau}=\left\{\begin{array}{c c l}u|_\tau, &&\text{if }\length(\sigma)=0,\\ 0, &&\text{otherwise.}\end{array}\right.
	\]
	A routine check shows that $\varphi^\bullet$ is a chain map.

	We define a bicomplex $\mathcal{L}^{\bullet,\bullet}$ by
	\[
	\mathcal{L}^{p,q}=\left\{\begin{matrix}\mathcal{S}^q(\mathbf{C};F) &\text{ if }p=0,\\0, &\text{ otherwise}\end{matrix}\right.
	\]
	and we let $d^h_{\mathcal{L}}=0$, $d^v_{\mathcal{L}}=0$ on the non-zero columns, and $d^v_{\mathcal{L}}=d_{\mathcal{S}^\bullet(\mathbf{C};F)}$ on the $0$-th column.

	Recall the bicomplex $\mathcal{K}_\xi^{p,q}=\mathcal{S}^p(\mathbf{B};\mathcal{S}^q)$ defined in Section \ref{sec:bicomplex}. We want to show that $\varphi$ induces a morphism of these two bicomplexes. To that effect, we need three facts:
	\begin{enumerate}
		\item First, it is clear that $\varphi(\mathcal{S}^q(\mathbf{C};F))\subseteq \mathcal{S}^0(\mathbf{B};\mathcal{S}^q)$.
		\item Second, we need $\varphi$ to induce a chain map on the vertical complexes. This is the zero map for $p\not=0$. Consider the diagram fragment
		\[\begin{tikzpicture}[scale=1.5,auto=center,tips=proper]
			
			\node (3) at (0,2)  {$\mathcal{S}^q(\mathbf{C};F)$};
			\node (4) at (0,3) {$\mathcal{S}^{q+1}(\mathbf{C};F)$};
			\node (8) at (3,2) {$\mathcal{S}^0(\mathbf{B};\mathcal{S}^q)$};
			\node (9) at (3,3) {$\mathcal{S}^0(\mathbf{B};\mathcal{S}^{q+1})$};
			
			\draw[->,tips=proper]
			(3) edge node[auto] {$d$} (4)
			(8) edge node[auto] {$d^v$} (9)
			(3) edge node[auto] {$\varphi$} (8)
			(4) edge node[auto] {$\varphi$} (9);
		\end{tikzpicture}\]
		We want to show that $d^v\varphi=\varphi d$. Let $u\in \mathcal{S}^0(\mathbf{B};\mathcal{S}^{q+1})$, $x\in \mathbf{B}$, $y_0\leq\cdots\leq y_{q+1}\in \mathbf{C}$. 
		\begin{align*}
			d^v\varphi u|_{x,y_0\leq\cdots\leq y_{q+1}}&=\sum_{i=0}^{q+1}\varphi u|_{x,y_0\leq\cdots\leq \hat{y}_i\leq\cdots\leq y_{q+1}}=\sum_{i=0}^{q+1} u|_{x,y_0\leq\cdots\leq \hat{y}_i\leq\cdots\leq y_{q+1}}\\
			&=du|_{y_0\leq\cdots\leq y_{q+1}}=\varphi du|_{x,y_0\leq\cdots\leq y_{q+1}}.
		\end{align*}
		
		Therefore $\varphi$ induces a chain map on vertical complexes.
		
		\item Finally, we need $\varphi$ to induce chain maps on horizontal complexes. Consider the diagram 
		\[
		\begin{tikzpicture}[scale=2.2,auto=center,tips=proper]
			
			\node (1) at (0,.7) {$\cdots$};
			\node (2) at (1,.7)  {$0$};
			\node (3) at (2,.7)  {$\mathcal{S}^q(\mathbf{C};F)$};
			\node (4) at (3,.7) {$0$};
			\node (5) at (4,.7) {$\cdots$};
			\node (6) at (0,0) {$\cdots$};
			\node (7) at (1,0) {$0$};
			\node (8) at (2,0) {$\mathcal{S}^0(\mathbf{B};\mathcal{S}^q)$};
			\node (9) at (3,0) {$\mathcal{S}^1(\mathbf{B};\mathcal{S}^q)$};
			\node (10) at (4,0) {$\cdots$};
			
			\draw[->,tips=proper]
			(1) edge (2)
			(2) edge (3)
			(3) edge (4)
			(4) edge (5)
			(6) edge (7)
			(7) edge (8)
			(8) edge (9)
			(9) edge (10)
			(2) edge node[auto] {$0$} (7)
			(3) edge node[auto] {$\varphi$} (8)
			(4) edge node[auto] {$0$} (9);
		\end{tikzpicture}\]
		This is just an instance of the map $\psi$ with $A=\mathcal{S}^q(\mathbf{C};F)$.
	\end{enumerate}
	
	Now consider the two spectral sequences $E$ and $E'$ associated to the bicomplexes $\mathcal{L}^{\bullet,\bullet}$ and $\mathcal{K}^{\bullet,\bullet}_\xi$, respectively. The morphism of bicomplexes $\varphi$ induces a morphism $E\to E'$ of spectral sequences. Note also that both $E$ and $E'$ are bounded below. We have $E_1^{p,q}=0$ if $p\not=0$ and $E_1^{0,q}=H^q(\mathbf{C};F)$, while $E'_1{}^{p,q}=\mathcal{S}^p(\mathbf{B};\mathcal{H}^q_{fib})$.

	As in the case of a constant bundle, the induced maps $\varphi$ are quasi-isomorphisms on the horizontal complexes. This means that $\varphi$ induces isomorphisms on the second pages of $E$ and $E'$. By the mapping lemma, we have an induced isomorphism $\varphi:E_\infty^{p,q}\to E'^{p,q}_\infty$.
	
	By the above, the construction of the total complex of a bicomplex, and Proposition \ref{Spectral}, we can conclude that $\varphi$ gives an isomorphism $\varphi:H^\bullet(\mathbf{C};F)\to H^\bullet T^\bullet_\xi$.
	\end{proof}

\section{Long exact sequence in the cohomology of the total complex}\label{sec:long1}
If we have a poset bundle $\xi:\mathbf{B}\to\mathbf{Sh}$ and a subcategory $\mathbf C$  of $\mathbf{B}$, then we will denote the chain complex $T^\bullet_{\xi_\mathbf{C}}$ (recall \S \ref{sec:bicomplex}) by just $T^\bullet_\mathbf{C}$. Below we headline the main result of this section and leave the proof until we have built up the required machinery.

\begin{theorem}\label{LES1}
	Let $\xi:\mathbf{B}\to\mathbf{Sh}$ be a poset bundle of sheaves with $\mathbf{B}$ an admissible poset for $x\succ 0$. Then there is a long exact sequence
	\begin{align*}
		\cdots\to H^{n-1} T^\bullet_{\mathbf{B}_{\not\geq x}}\to H^n T^\bullet_\xi\to H^n T^\bullet_{\mathbf{B}_{\geq x}}\oplus H^n T^\bullet_{\mathbf{B}_{\not\geq x}}\to H^{n} T^\bullet_{\mathbf{B}_{\not\geq x}}\to H^{n+1} T^\bullet_\xi\to \cdots
	\end{align*}
\end{theorem}

We will need to leverage the admissibility condition in the theorem to establish the connection between the total complex of the whole sheaf and those of the two smaller parts $\mathbf{B}_{\geq x}$ and $\mathbf{B}_{\not\geq x}$, determined by the element $x\succ 0$. Recall that we assume all the $\mathbf{E}_y$ are posets.

Where possible, we will use $x$'s to refer to objects in $\mathbf{B}_{\not\geq x}$ and $z$'s to refer to objects of $\mathbf{B}_{\geq x}$. We can write down explicitly what $ T^n_\xi$, $ T^n_{\mathbf{B}_{\geq x}}$, and $ T^n_{\mathbf{B}_{\not\geq x}}$ are:
\begin{align*}
	T^n_\xi=\bigoplus_{p+q=n}\, \prod_{\substack{x_0\leq\cdots\leq x_p\in \mathbf{B}\\y_0\leq\cdots\leq y_q\in \mathbf{E}_{x_0}}} \kern-7mm F_{x_0}(y_0),\quad
	T^n_{\mathbf{B}_{\geq x}}=\bigoplus_{p+q=n}\, \prod_{\substack{z_0\leq\cdots\leq z_p\in \mathbf{B}_{\geq x}\\y_0\leq\cdots\leq y_q\in \mathbf{E}_{z_0}}} \kern-7mm F_{z_0}(y_0),\quad
	T^n_{\mathbf{B}_{\not\geq x}}=\bigoplus_{p+q=n}\, \prod_{\substack{x_0\leq\cdots\leq x_p\in \mathbf{B}_{\not\geq x}\\y_0\leq\cdots\leq y_q\in \mathbf{E}_{x_0}}} \kern-7mm F_{x_0}(y_0).
\end{align*}

Define the quotient map
\[
\rho: T^n_\xi\to T^n_{\mathbf{B}_{\geq x}}\oplus T^n_{\mathbf{B}_{\not\geq x}}
\]
by setting to $0$ any coordinate corresponding to a sequence $x_0\leq\cdots\leq x_p\in \mathbf{B}$ that has objects in both $\mathbf{B}_{\geq x}$ and $\mathbf{B}_{\not\geq x}$. Explicitly, if $u\in T^{p+q}_\xi$, $\sigma=x_0\leq\cdots\leq x_p\in \mathbf{B}_{\geq x}$ or $\mathbf{B}_{\not\geq x}$, and $\tau\in \mathbf{E}_{x_0}$, then $\rho u|_{\sigma,\tau}=u|_{\sigma,\tau}$.

To see that $\rho$ is a chain map, let $\{x_0,\ldots,x_p\}\subseteq \mathbf{B}_{\geq x}$. We have
\begin{align*}
	\rho du|_{\sigma,\tau}=du|_{\sigma,\tau}&=\sum_{i=0}^p (-1)^i u|_{\sigma_i,\tau}+(-1)^{p+q}\sum_{j=0}^q (-1)^j u|_{\sigma,\tau_j}\\
	&=\sum_{i=0}^p (-1)^i \rho u|_{\sigma_i,\tau}+(-1)^{p+q}\sum_{j=0}^q (-1)^j \rho u|_{\sigma,\tau_j}=d\rho u|_{\sigma,\tau}.
\end{align*}
The calculation is analogous if $\{x_0,\ldots,x_p\}\subseteq\mathbf{B}_{\not\geq x}$. Therefore $\rho$ is a chain map. It is also clearly surjective, so we have a short exact sequence
\[
0\to M^\bullet\to  T^\bullet_\xi\to T^\bullet_{\mathbf{B}_{\geq x}}\oplus  T^\bullet_{\mathbf{B}_{\not\geq x}}\to 0
\]
for a particular chain complex $M^\bullet$. We describe $M^\bullet$ explicitly:
\[
M^n=\bigoplus_{p+q=n}\, \prod_{x_0\leq\cdots\leq x_p}\, \prod_{y_0\leq\cdots\leq y_q\in \mathbf{E}_{x_0}} F_{x_0}(y_0),
\]
where $x_0\in \mathbf{B}_{\not\geq x},x_p\in \mathbf{B}_{\geq x}$. We can rewrite $M^\bullet$ to pay attention to how many of the $x_i$'s are in $\mathbf{B}_{\not\geq x}$ and how many are in $\mathbf{B}_{\geq x}$:
\[
M^n=\bigoplus_{s+t+q=n}\, \prod_{x_0\leq\cdots\leq x_s\leq z_0\leq\cdots\leq z_{t-1}}\, \prod_{y_0\leq\cdots\leq y_q\in \mathbf{E}_{x_0}} F_{x_0}(y_0),
\]
where $x_i\in\mathbf{B}_{\not\geq x},z_i\in \mathbf{B}_{\geq x},s\geq 0,t\geq 1$.

\begin{proposition}\label{first}
	Let $\xi:\mathbf{B}\to\mathbf{Sh}$ be a poset bundle of sheaves with $\mathbf{B}$ an admissible poset for $x\succ 0$. If $M^\bullet$ is as above, there is a chain map $\varphi_1:{T}^{n-1}_{\mathbf{B}_{\not\geq x}}\to M^n$ that induces an isomorphism in cohomology.
\end{proposition}

\begin{proof}
	In an attempt to keep the notation less cluttered, write $K^n= T^{n-1}_{\mathbf{B}_{\not\geq x}}$.
	
	We define the chain map $\varphi_1:K^n\to M^n$, which will extend to a morphism of filtered complexes. By showing that $\varphi_1$ induces isomorphisms on the first pages of the two spectral sequences associated to the two filtrations, the Mapping Lemma implies that $\varphi_1$ is a quasi-isomorphism.
	
	Let $\sigma=x_0\leq\cdots\leq x_s\leq z_0\leq\cdots\leq z_{t-1}$ be a sequence in $\mathbf{B}$ with $x_i\in \mathbf{B}_{\not\geq x}$, $z_i\in \mathbf{B}_{\geq x}$, $s\geq 0$, $t\geq 1$. Denote $\sigma'=x_0\leq\cdots\leq x_s$. Also let $\tau=y_0\leq\cdots\leq y_q$ be a sequence in $\mathbf{E}_{x_0}$. Now if $s+t+q=n$, we define $\varphi_1:K^n\to M^n$ by
	\[
	\varphi_1 u|_{\sigma,\tau}=\left\{\begin{array}{c c l}(-1)^q u|_{\sigma',\tau} && \text{if }t=1\\ 0 && \text{otherwise.}\end{array}\right.
	\]
	Intuitively, $\varphi_1$ acts like the map $\varphi$ in Proposition \ref{alpha} on the portion of $M^\bullet$ that matches $ T^\bullet_{\mathbf{B}_{\not\geq x}}$.  A routine check shows that $\varphi_1$ is a chain map.
	
	Now we define filtrations of $M^\bullet$ and $K^\bullet$:
	\begin{align*}
		\mathcal F^pM^n=\{u\in M^n : u|_{\sigma,\tau}\not=0\Rightarrow \sigma=x_0\leq\cdots\leq x_s&\leq z_0\leq\cdots\leq z_{t-1}\text{ with }s\geq p\},
	\end{align*}
	\[
	\mathcal{J}^pK^n=\{u\in K^n : u|_{\sigma,\tau}\not=0\Rightarrow\sigma=x_0\leq\cdots\leq x_s\text{ with }s\geq p\}.
	\]
	
	We want to use the Mapping Lemma for these two filtrations, so the next step is establishing all the assumptions of the lemma. We prove them for $\mathcal F$ with the arguments for $\mathcal J$ being analogous.
	\begin{description}%[leftmargin=*]
		\item[\rm ($\mathcal F$ is a filtration)] It is clear from the definition of $\mathcal F$ that $\mathcal F^{p+1}M^n\subseteq \mathcal F^{p}M^n$ for each $p$ and $n$. Remains to show that $\mathcal F^pM^\bullet$ is a complex for each $p$. Let $\sigma=x_0\leq\cdots\leq x_s\leq z_0\leq\cdots\leq z_{t-1}$ with $s<p$ and $u\in \mathcal F^pM^n$. Then for any sequence $\tau \in \mathbf{E}_{x_0}$ (of appropriate length $q$) we have
		\begin{align*}
			du|_{\sigma,\tau}&=\sum_{i=0}^s(-1)^i u|_{\sigma_i,\tau}+(-1)^{s+1}\sum_{k=0}^{t-1}(-1)^k u|_{\sigma_{s+k},\tau}++(-1)^{s+t+q}\sum_{\ell=0}^q(-1)^\ell u|_{\sigma,\tau_\ell}.
		\end{align*}
		The summands in the first sum correspond to $x$-sequences of length $s-1<p$, while the summands in the other two sums correspond to $x$-sequences of length $s<p$. All those coordinates are $0$ in $u\in \mathcal F^pM^n$, so $d$ induces a differential on $\mathcal F^pM^\bullet$.
		
		\item [\rm ($\mathcal F$ is convergent below)] Observe that $\mathcal F^0M^n=M^n$, since $M^n$ does not have any coordinates corresponding to sequences in $\mathbf{B}$ not containing elements of $\mathbf{B}_{\not\geq x}$.
  
		\item [\rm ($\mathcal F$ is bounded above)] Observe that $\mathcal F^nM^n=0$, since we need $s+t+q=n$ and $t\geq 1$.
  
		\item [\rm ($\varphi_1$ is a morphism of filtrations)] Let $u\in \mathcal{J}^pK^n$, set $\sigma=x_0\leq\cdots\leq x_s\leq z$ and $\tau=y_0\leq\cdots\leq y_q$. First suppose $s+q+1\not=n$. The potentially non-zero coordinates of $\varphi_1 u|_{\sigma,\tau}$ correspond to sequences of combined length satisfying $s+q\not=n-1$, so they are also $0$. Now suppose $s<p$. Again, the potentially non-zero coordinates of $\varphi_1 u|_{\sigma,\tau}$ correspond to $x$-sequences of length $s<p$, so are also $0$. Thus $\varphi_1(\mathcal{J}^pK^n)\subseteq \mathcal F^pM^n$.
		
		To see that $\varphi_1$ induces chain maps $\mathcal{J}^pK^\bullet\to \mathcal F^pM^\bullet$ for every $p$, note that we already know that $d\varphi_1=\varphi_1 d$ and that $\varphi_1(\mathcal{J}^pK^n)\subseteq \mathcal F^pM^n$.
	\end{description}
	
	Let $E, E'$ be the spectral sequences associated to the filtrations $\mathcal F,\mathcal J$, respectively. We have
	\begin{align*}
		E_0^{p,q}&=\frac{\mathcal F^pM^{p+q}}{\mathcal F^{p+1}M^{p+q}}=\{u\in M^{p+q} : u|_{\sigma,\tau}\not=0\Rightarrow \sigma=x_0\leq\cdots\leq x_p\leq z_0\leq\cdots\leq z_{t-1}\},\\
		E_0'^{p,q}&=\frac{\mathcal{J}^pK^{p+q}}{\mathcal{J}^{p+1}K^{p+q}}=\{u\in K^{p+q} : u|_{\sigma,\tau}\not=0\Rightarrow\sigma=x_0\leq\cdots\leq x_p\}.
	\end{align*}
	
	The vertical differentials in $E_0$ are given by
	\begin{align*}
		du|_{x_0\leq\cdots\leq x_p\leq z_0\leq\cdots\leq z_{t-1},y_0\leq\cdots\leq y_{q-t}}&=(-1)^{p+1}\sum_{i=0}^{t-1}(-1)^i u|_{x_0\leq\cdots\leq x_p\leq z_0\leq\cdots\leq \hat{z}_i\leq\cdots\leq z_{t-1},y_0\leq\cdots\leq y_{q-t}}+\\
		&\quad+(-1)^{p+q}\sum_{\ell=0}^{q-t}(-1)^\ell u|_{x_0\leq\cdots\leq x_p\leq z_0\leq\cdots\leq z_{t-1},y_0\leq\cdots\leq \hat{y}_\ell\leq\cdots\leq y_{q-t}}
	\end{align*}
	and the vertical differentials in $E'_0$ are given by
	\begin{align*}
		du|_{x_0\leq\cdots\leq x_p,y_0\leq\cdots\leq y_{q}}=(-1)^{p+q}\sum_{\ell=0}^{q}(-1)^\ell u|_{x_0\leq\cdots\leq x_p,y_0\leq\cdots\leq \hat{y}_\ell\leq\cdots\leq y_{q}}.
	\end{align*}
	
	Using the notation from Definition \ref{admissible} we can thus rewrite
	\[
	E_0^{p,\bullet}=\prod_{x_0\leq\cdots\leq x_p}(-1)^{p+1} T^{\bullet-1}_{\mathbf{B}_{\geq x}^{\geq x_p}\times (\mathbf{E}_{x_0},F_{x_0})}\quad\text{and}\quad
	E'_0{}^{p,\bullet}=\prod_{x_0\leq\cdots\leq x_p}(-1)^{p+q}\mathcal{S}^{\bullet-1}(\mathbf{E}_{x_0};F_{x_0}).
	\]
	
	Now note that $\varphi_1$ acts as the product over all $p$-long $x$-sequences in $\mathbf{B}_{\not\geq x}$ of the maps in Proposition \ref{alpha}, since $\mathbf{B}$ is an admissible poset and thus the subposet $\mathbf{B}_{\geq x}^{\geq x_p}$ has a unique minimum. This means that $\varphi_1:E'_0{}^{p,\bullet}\to E_0^{p,\bullet}$ is a quasi-isomorphism and thus
	\[
	E'_1{}^{p,q}=H^p(E'_0{}^{p,\bullet})\overset{\varphi^\bullet_1}{\cong}H^p(E_0^{p,\bullet})=E_1^{p,q}.
	\]
	
	The Mapping Lemma then implies that
	\[
	\varphi^\bullet_1:H^{n-1} T^{\bullet}_{\mathbf{B}_{\not\geq x}}\cong H^n(M^\bullet).\qedhere
	\]
\end{proof}

We can now easily complete the proof of the theorem, headlined at the start of this section.

\begin{proof}[Proof of Theorem \ref{LES1}]
	We have the short exact sequence from before 
	\[
	0\to M^\bullet\to  T^\bullet_\xi\to T^n_{\mathbf{B}_{\geq x}}\oplus T^n_{\mathbf{B}_{\not\geq x}}\to 0,
	\]
	from which we get a long exact sequence in homology
	\begin{align*}
		\cdots\to H^{n-1} T^{\bullet}_{\mathbf{B}_{\geq x}}\oplus H^{n-1} T^\bullet_{\mathbf{B}_{\not\geq x}}\to H^nM^\bullet\to
		H^n T^\bullet_\xi\to H^n T^\bullet_{\mathbf{B}_{\geq x}}&\oplus H^n T^\bullet_{\mathbf{B}_{\not\geq x}}\to\\
		&\to H^{n+1}M^\bullet\to\cdots
	\end{align*}
	Replacing the occurrences of $H^n M^\bullet$ with $H^{n-1} T^{\bullet}_{\mathbf{B}_{\not\geq x}}$ and the maps around those occurrences with the appropriate compositions with $\varphi^\bullet_1$ and $\varphi^\bullet_1{}^{-1}$ gives the required long exact sequence.
\end{proof}

\section{Long exact sequence in sheaf cohomology}\label{sec:long2}
We now repeat this procedure for the cochain complex of the total sheaf $(\mathbf{E}_\xi,F_\xi)$. The story is fairly similar to that of the previous section, so we are a little briefer. Again, we headline the main result, with the proof delayed until the end of the section.
\begin{theorem}\label{LES2}
	Let $\xi:\mathbf{B}\to\mathbf{Sh}$ be a poset bundle of sheaves with $\mathbf{B}$ an admissible poset. Then there is a long exact sequence
	\begin{align*}
		\cdots\to H^{n-1}(\mathbf{E}_{\mathbf{B}_{\not\geq x}};&F_{\mathbf{B}_{\not\geq x}})\to H^n (\mathbf{E}_\xi;F_\xi)\to H^n(\mathbf{E}_{\mathbf{B}_{\geq x}};F_{\mathbf{B}_{\geq x}})\oplus H^n (\mathbf{E}_{\mathbf{B}_{\not\geq x}};F_{\mathbf{B}_{\not\geq x}})\to \cdots
	\end{align*}
\end{theorem}

Where possible, we will use $x$'s to refer to objects in $\mathbf{E}_{\mathbf{B}_{\not\geq x}}$ and $z$'s to refer to objects of $\mathbf{E}_{\mathbf{B}_{\geq x}}$. We can write down explicitly what $\mathcal{S}^n(\mathbf{E}_\xi;F_\xi)$, $\mathcal{S}^n(\mathbf{E}_{\mathbf{B}_{\geq x}};F_{\mathbf{B}_{\geq x}})$, and $\mathcal{S}^n(\mathbf{E}_{\mathbf{B}_{\not\geq x}};F_{\mathbf{B}_{\not\geq x}})$ are:
\begin{align*}
	\mathcal{S}^n(\mathbf{E}_\xi;F_\xi)=\kern-5mm\prod_{x_0\leq\cdots\leq x_n\in \mathbf{E_\xi}} \kern-5mm F_\xi(x_0),\quad
	\mathcal{S}^n(\mathbf{E}_{\mathbf{B}_{\geq x}};F_{\mathbf{B}_{\geq x}})= \kern-7mm \prod_{z_0\leq\cdots\leq z_n\in \mathbf{E}_{\mathbf{B}_{\geq x}}}  \kern-7mm F_\xi(x_0),\quad
	\mathcal{S}^n(\mathbf{E}_{\mathbf{B}_{\not\geq x}};F_{\mathbf{B}_{\not\geq x}})= \kern-7mm \prod_{x_0\leq\cdots\leq x_n\in \mathbf{E}_{\mathbf{B}_{\not\geq x}}} \kern-7mm  F_\xi(x_0).
\end{align*}

Define another quotient map
\[
\rho:\mathcal{S}^n(\mathbf{E}_\xi;F_\xi)\to \mathcal{S}^n(\mathbf{E}_{\mathbf{B}_{\geq x}};F_{\mathbf{B}_{\geq x}})\oplus \mathcal{S}^n(\mathbf{E}_{\mathbf{B}_{\not\geq x}};F_{\mathbf{B}_{\not\geq x}})
\]
by setting to $0$ any coordinate corresponding to a sequence $x_0\leq\cdots\leq x_n$ in $\mathbf{E}_\xi$ that has objects in both $\mathbf{E}_{\mathbf{B}_{\geq x}}$ and $\mathbf{E}_{\mathbf{B}_{\not\geq x}}$. This is a chain map by an analogous argument to the one for the quotient before Proposition \ref{first}.

The map $\rho$ is clearly surjective, so we have an short exact sequence
\[
0\to N^\bullet\to \mathcal{S}^\bullet(\mathbf{E}_\xi;F_\xi)\to \mathcal{S}^n(\mathbf{E}_{\mathbf{B}_{\geq x}};F_{\mathbf{B}_{\geq x}})\oplus \mathcal{S}^n(\mathbf{E}_{\mathbf{B}_{\not\geq x}};F_{\mathbf{B}_{\not\geq x}})\to 0
\]
for a particular chain complex $N^\bullet$.

We describe $N^\bullet$ explicitly:
\[
N^n=\prod_{x_0\leq\cdots\leq x_p\leq z_0\leq\cdots\leq z_{n-p-1}}F_\xi(x_0),
\]
where $x_i\in \mathbf{E}_{\mathbf{B}_{\not\geq x}},z_i\in \mathbf{E}_{\mathbf{B}_{\geq x}},p\geq 0,n-p\geq1$.

\begin{proposition}\label{second}
	Let $\xi:\mathbf{B}\to\mathbf{Sh}$ be a poset bundle of sheaves with $\mathbf{B}$ an admissible poset for $x\succ0$. If $N^\bullet$ is as above, there is a chain map
	\[
	\varphi_2:\mathcal{S}^{n-1}(\mathbf{E}_{\mathbf{B}_{\not\geq x}};F_{\mathbf{B}_{\not\geq x}})\to N^n
	\]
	that induces an isomorphism in cohomology.
\end{proposition}

\begin{proof}
	We define a filtration $\mathcal J$ of $N^\bullet$:
	\begin{align*}
		\mathcal J^pN^n=\{u\in N^n : u|_\sigma\not=0\Rightarrow \sigma=x_0\leq\cdots\leq x_s&\leq z_0\leq\cdots\leq z_{n-s-1},\text{ with }s\geq p\}.
	\end{align*}
	The proof that this is a filtration is analogous to the proofs of the filtrations from  Proposition \ref{first}.
	
	Let $E$ be the spectral sequence associated to the filtration $\mathcal J$ of $N$. We have
	\begin{align*}
		E_0^{p+q}&=\frac{\mathcal J^pN^{p+q}}{\mathcal J^{p+1}N^{p+q}}=\{u\in B^n : u|_\sigma\not=0\Rightarrow \sigma=x_0\leq\cdots\leq x_p\leq z_0\leq\cdots\leq z_{q-1}\}.
	\end{align*}
	
	\noindent The vertical differentials in $E_0$ are given by
	\[
	du|_{x_0\leq\cdots\leq x_p\leq z_0\leq\cdots\leq z_{q-1}}=(-1)^{p+1}\sum_{i=0}^{q-1} (-1)^i u|_{x_0\leq\cdots\leq x_p\leq z_0\leq\cdots\leq \hat z_{i}\leq\cdots\leq z_{q-1}}.
	\]
	
	\noindent We can thus write
	\[
	E_0^{p,\bullet}=\prod_{x_0\leq\cdots\leq x_p} (-1)^{p+1}\mathcal{S}^{\bullet-1}(\{z\in \mathbf{E}_{\mathbf{B}_{\geq x}}\mid z\geq x_p\},\Delta F_\xi(x_0)).
	\]
	
	But the $\mathcal{S}$ complex on the right is of a poset with a constant sheaf. By Lemma \ref{gluedminimal} the underlying poset has a unique minimum, so
	\begin{align*}
		E_1^{p,q}=H^q E_0^{p,\bullet}&=\left\{\begin{array}{c c l}\underset{x_0\leq\cdots\leq x_p}{\prod} (-1)^{p+1} F_\xi(x_0)&&\text{if }q=1,\\0&&\text{otherwise.}\end{array}\right.\\
		&=\left\{\begin{array}{c c l}(-1)^n \mathcal{S}^{n-1}(\mathbf{E}_{\mathbf{B}_{\not\geq x}};F_{\mathbf{B}_{\not\geq x}})&&\text{if }q=1,\\0&&\text{otherwise.}\end{array}\right.
	\end{align*}
	
	\noindent So on the $E_1$ page we have the  single  $q=1$ row
	\[
	\cdots\rightarrow (-1)^n\mathcal{S}^{n-1}(\mathbf{E}_{\mathbf{B}_{\not\geq x}};F_{\mathbf{B}_{\not\geq x}})\rightarrow (-1)^{n+1}\mathcal{S}^n(\mathbf{E}_{\mathbf{B}_{\not\geq x}};F_{\mathbf{B}_{\not\geq x}})\rightarrow\cdots.
	\]
	
	The differential on this page is induced by the differential
	\[
	du|_{x_0\leq\cdots\leq x_p\leq z_0\leq\cdots\leq z_{q-1}}=\sum_{i=0}^p (-1)^i u|_{x_0\leq\cdots\leq \hat x_i\leq\cdots\leq x_p\leq z_0\leq\cdots\leq z_{q-1}},
	\]
	which, since it keeps the $z$-sequence constant, induces the following differential on the above row on the $E_1$ page:
	\[
	du|_{x_0\leq\cdots\leq x_p}=\sum_{i=0}^p (-1)^i u|_{x_0\leq\cdots\leq \hat x_i\leq\cdots\leq x_p}.
	\]
	
	\noindent Since $d (-d)=(-d) d=0$, $\ker(-d)=\ker d$, and $\im(-d)=\im d$, we have that the $E_2$ page is
	\[
	E_2^{p,q}\cong\left\{\begin{array}{c c l}H^{p+q-1}\mathcal{S}^\bullet(\mathbf{E}_{\mathbf{B}_{\not\geq x}};F_{\mathbf{B}_{\not\geq x}})&&\text{if }q=1,\\0&&\text{otherwise.}\end{array}\right.
	\]
	
	Then $E_2^{p,q}\cong E_\infty^{p,q}$ and so
	\[
	E\Rightarrow H^{n-1}\mathcal{S}^\bullet(\mathbf{E}_{\mathbf{B}_{\not\geq x}};F_{\mathbf{B}_{\not\geq x}})\cong N^n.
	\]
	In particular, this isomorphism is witnessed by a similar quasi-isomorphism to that in Proposition \ref{first}, namely $\varphi_2:\mathcal{S}^{n-1}(\mathbf{E}_{\mathbf{B}_{\not\geq x}};F_{\mathbf{B}_{\not\geq x}})\to N^n$ defined by
	\[
	\varphi_2u|_{x_0\leq\cdots\leq x_n}=\left\{\begin{array}{c c l}u|_{x_0\leq\cdots\leq x_{n-1}}&&\text{if }x_{n-1}\in  \mathbf{E}_{\mathbf{B}_{\not\geq x}},x_n\in \mathbf{E}_{\mathbf{B}_{\geq x}},\\ 0&&\text{otherwise.}\end{array}\right.\qedhere
	\]
\end{proof}

We can now, again, easily prove the headlined theorem.

\begin{proof}[Proof of Theorem \ref{LES2}]
	We have the short exact sequence from before
	\[
	0\to N^\bullet\to \mathcal{S}^\bullet(\mathbf{E}_\xi;F_\xi)\to\mathcal{S}^\bullet(\mathbf{E}_{\mathbf{B}_{\not\geq x}};F_{\mathbf{B}_{\not\geq x}})\oplus\mathcal{S}^\bullet(\mathbf{E}_{\mathbf{B}_{\geq x}};F_{\mathbf{B}_{\geq x}})\to 0,
	\]
	from which we get a long exact sequence in homology
	\begin{align*}
		\cdots\to H^{n-1}(\mathbf{E}_{\mathbf{B}_{\geq x}};F_{\mathbf{B}_{\geq x}})&\oplus H^{n-1}(\mathbf{E}_{\mathbf{B}_{\not\geq x}};F_{\mathbf{B}_{\not\geq x}})\to H^nN^\bullet\to H^n(\mathbf{E}_\xi;F_\xi)\to\\
		&\to H^n(\mathbf{E}_{\mathbf{B}_{\geq x}};F_{\mathbf{B}_{\geq x}})\oplus H^n(\mathbf{E}_{\mathbf{B}_{\not\geq x}};F_{\mathbf{B}_{\not\geq x}})\to H^{n+1}N^\bullet\to\cdots
	\end{align*}
	Replacing the occurrences of $H^n N^\bullet$ with $H^{n-1}(\mathbf{E}_{\mathbf{B}_{\not\geq x}};F_{\mathbf{B}_{\not\geq x}})$ and the maps around those occurrences with the appropriate compositions with $\varphi^\bullet_2$ and $\varphi^\bullet_2{}^{-1}$ gives the required long exact sequence.
\end{proof}
\section{The bicomplex and the total sheaf}\label{sec:main}

We have all the necessary prerequisites to prove the main theorem:

\begin{theorem}\label{MainTheorem}
	Let $\xi:\mathbf{B}\to \mathbf{Sh}$ be a poset bundle of sheaves with $\mathbf{B}$ a recursively admissible finite poset, and $(\mathbf{E}_\xi;F_\xi)$ the associated total sheaf. Then there is a spectral sequence
	\[
	E_2^{p,q}=H^p(\mathbf{B};\mathcal{H}^q_{fib}(\xi))\Rightarrow H^\bullet(\mathbf{E}_\xi;F_\xi).
	\]
\end{theorem}

\begin{proof}
	Proposition \ref{Spectral} gives us
	\[
	E_2^{p,q}=H^p(\mathbf{B};\mathcal{H}^q_{fib}(\xi))\Rightarrow H^\bullet T^\bullet_\xi,
	\]
	so it is enough to show that $H^\bullet T^\bullet_\xi\cong H^\bullet(\mathbf{E}_\xi,F_\xi)$. We will do this by induction on the size of $\mathbf{B}$. Recall the chain map $\omega:\mathcal{S}^\bullet(\mathbf{E}_\xi;F_\xi)\to {T}^\bullet_\xi$ from Section \ref{sec:traversals}:
	\[
	\omega u|_{\sigma,\tau}=(-1)^{\varsigma(q)}\sum_z (-1)^{m(z)}u|_z,
	\]
	where the sum is taken over all traversals $z$ of the grid of $(\sigma,\tau)$. We have two short exact sequences from Theorems \ref{first} and \ref{second}. The map $\omega$ gives a morphism of these short exact sequences
	\[\begin{tikzpicture}[scale=1,auto=center,tips=proper]
		\node (1) at (2.5,0) {$0$};
		\node (2) at (4,0)  {$M^n$};
		\node (3) at (6,0)  {${T}^n_\xi$};
		\node (4) at (11,0) {${T}^n_{\mathbf{B}_{\not\geq x}}\oplus{T}^n_{\mathbf{B}_{\geq x}}$};
		\node (5) at (15,0) {$0$};
		\node (6) at (2.5,2) {$0$};
		\node (7) at (4,2) {$N^n$};
		\node (8) at (6,2) {$\mathcal{S}^n(\mathbf{E}_\xi;F_\xi)$};
		\node (9) at (11,2) {$\mathcal{S}^n(\mathbf{E}_{\mathbf{B}_{\not\geq x}};F_{\mathbf{B}_{\not\geq x}})\oplus\mathcal{S}^n(\mathbf{E}_{\mathbf{B}_{\geq x}};F_{\mathbf{B}_{\geq x}})$};
		\node (10) at (15,2) {$0$};
		
		\draw[->,tips=proper]
		(1) edge (2)
		(2) edge node[auto] {$\varepsilon$} (3)
		(3) edge node[auto] {$\pi$} (4)
		(4) edge (5)
		(6) edge (7)
		(7) edge node[auto] {$\varepsilon$} (8)
		(8) edge node[auto] {$\pi$} (9)
		(9) edge (10)
		(7) edge node[auto] {$\omega'$} (2)
		(8) edge node[auto] {$\omega$} (3)
		(9) edge node[auto] {$\omega\oplus\omega$} (4);
	\end{tikzpicture}
	\]
	where the maps $\varepsilon$ are the injections and the maps $\pi$ the projections of the respective modules. The map $\omega'$ is the restriction of $\omega$ to the subcomplexes $N^n$ and $M^n$. We need to check the commutativity of the two squares.
	\begin{description}
		\item [\rm (Left square).] The maps $\varepsilon$ are just injections, so we have
		\begin{align*}
			\varepsilon\omega u|_{\sigma,\tau}&=\omega u|_{\sigma,\tau}=(-1)^{\varsigma(q)}\sum_z (-1)^{m(z)}u|_z=(-1)^{\varsigma(q)}\sum_z (-1)^{m(z)}\varepsilon u|_z=\omega\varepsilon u|_{\sigma,\tau}.
		\end{align*}
		\item [\rm (Right square).] Similarly, the maps $\pi$ are projections, so
		\begin{align*}
			\pi\omega u|_{\sigma,\tau}&=\omega u|_{\sigma,\tau}=(-1)^{\varsigma(q)}\sum_z (-1)^{m(z)}u|_z=(-1)^{\varsigma(q)}\sum_z (-1)^{m(z)}\pi u|_z=\omega\pi u|_{\sigma,\tau}.
		\end{align*}
	\end{description}
	
	The naturality of the homology functor then gives a morphism of long exact sequences, which contains the  commutative diagram in Figure \ref{CommutativeDiagram1}.
	
	\begin{figure}{\begin{tikzpicture}[scale=1,auto=center,tips=proper]
				\node (1) at (8,0) {$H^{n-1}{T}^\bullet_{\mathbf{B}_{\not\geq x}}\oplus H^{n-1}{T}^\bullet_{\mathbf{B}_{\geq x}}$};
				\node (2) at (8,2)  {$H^nM^\bullet$};
				\node (3) at (8,4)  {$H^n {T}^\bullet_\xi$};
				\node (4) at (8,6) {$H^n{T}^\bullet_{\mathbf{B}_{\not\geq x}}\oplus H^n{T}^\bullet_{\mathbf{B}_{\geq x}}$};
				\node (5) at (8,8) {$H^{n+1}M^\bullet$};
				\node (6) at (0,0) {$H^{n-1}(\mathbf{E}_{\mathbf{B}_{\not\geq x}};F_{\mathbf{B}_{\not\geq x}})\oplus H^{n-1}(\mathbf{E}_{\mathbf{B}_{\geq x}};F_{\mathbf{B}_{\geq x}})$};
				\node (7) at (0,2) {$H^nN^\bullet$};
				\node (8) at (0,4) {$H^n(\mathbf{E}_\xi;F_\xi)$};
				\node (9) at (0,6) {$H^n(\mathbf{E}_{\mathbf{B}_{\not\geq x}};F_{\mathbf{B}_{\not\geq x}})\oplus H^n(\mathbf{E}_{\mathbf{B}_{\geq x}};F_{\mathbf{B}_{\geq x}})$};
				\node (10) at (0,8) {$H^{n+1}N^\bullet$};
				
				\draw[->,tips=proper]
				(1) edge node[auto] {$\delta$} (2)
				(2) edge node[auto] {$\varepsilon^\bullet$} (3)
				(3) edge node[auto] {$\pi^\bullet$} (4)
				(4) edge node[auto] {$\delta$} (5)
				(6) edge node[auto] {$\delta$} (7)
				(7) edge node[auto] {$\varepsilon^\bullet$} (8)
				(8) edge node[auto] {$\pi^\bullet$} (9)
				(9) edge node[auto] {$\delta$} (10)
				(7) edge node[auto] {$\omega'{}^\bullet$} (2)
				(8) edge node[auto] {$\omega^\bullet$} (3)
				(9) edge node[auto] {$\omega^\bullet\oplus\omega^\bullet$} (4)
				(6) edge node[auto] {$\omega^\bullet\oplus\omega^\bullet$} (1)
				(10) edge node[auto] {$\omega'{}^\bullet$} (5);
		\end{tikzpicture}}
		\caption{A portion of the commutative diagram given by the morphism of short exact sequences.}
		\label{CommutativeDiagram1}
	\end{figure}
	
	Recall from  Propositions \ref{first} and \ref{second} the quasi-isomorphisms 
	\[
	\varphi_1:{T}^{n-1}_{\mathbf{B}_{\not\geq x}}\to M^n\,\text{ and }\,\varphi_2:\mathcal{S}^{n-1}(\mathbf{E}_{\mathbf{B}_{\not\geq x}};F_{\mathbf{B}_{\not\geq x}})\to M^n.
	\]
	
	\begin{claim}The following diagram commutes
		\[
		\begin{tikzpicture}[scale=1.4,auto=center,tips=proper]
			every edge/.style = {draw, -Latex}
			
			\node (1) at (0,0) {${T}^{n-1}_{\mathbf{B}_{\not\geq x}}$};
			\node (2) at (0,1.5)  {$\mathcal{S}^{n-1}(\mathbf{E}_{\mathbf{B}_{\not\geq x}};F_{\mathbf{B}_{\not\geq x}})$};
			\node (3) at (3.5,1.5)  {$N^n$};
			\node (4) at (3.5,0) {$M^n$};
			
			\draw[->,tips=proper]
			(2) edge node[auto] {$\omega$} (1)
			(2) edge node[above] {$\varphi_2$} (3);
			
			\draw[->,tips=proper]
			(3) edge node[auto] {$\omega'$} (4)
			(1) edge node[above] {$\varphi_1$} (4);
		\end{tikzpicture}\]
	\end{claim}
	
	\begin{proof}[Proof of claim]
		Let $u\in \mathcal{S}^{n-1}(\mathbf{E}_{\mathbf{B}_{\not\geq x}};F_{\mathbf{B}_{\not\geq x}})$. Suppose 
		\[
		\sigma=x_0\leq\cdots\leq x_s\leq z_0\leq\cdots\leq z_{t-1},\, \tau=y_0\leq\cdots\leq y_q
		\]
		with $s+t+q=n$. If $t> 1$, it is clear that $\varphi_1\omega u|_{\sigma,\tau}=0=\omega'\varphi_2u|_{\sigma,\tau}$, since each summand of $\omega'\varphi_2u|_{\sigma,\tau}$ is $0$ under $\varphi_2$. If $t=1$, let $\sigma'=x_0\leq\cdots\leq x_s$. Then we have
		\begin{align*}
			\omega'\varphi_2 u|_{\sigma,\tau}&=(-1)^{\varsigma(q)}\sum_{z'}(-1)^{m(z')}\varphi_2 u|_{z'},
		\end{align*}
		where the  sum is taken over  the traversals $z'$ of $(\sigma,\tau)$.
		
		Pick a traversal $z'$ of $(\sigma,\tau)$. We zoom in on the top right of the grid of $(\sigma,\tau)$:
		\[\begin{tikzpicture}[scale=1.4,auto=center,tips=proper]
			every edge/.style = {draw, -Latex}
			
			\node (1) at (1,0) {$y'_0$};
			\node (2) at (0,1)  {$y'_1$};
			\node (3) at (1,1)  {$y'_2$};
			\node (4) at (-0.5,1) {$\cdots$};
			\node (5) at (0,0) {\reflectbox{$\ddots$}};
			\node (6) at (1,-0.5) {$\vdots$};
			
			\draw[->,tips=proper]
			(1) edge (3)
			(2) edge (3);
		\end{tikzpicture}\]
		Note that $y'_0,y'_2\in \mathbf{E}_{z_0}$. If $z'$ passes through $y'_0$, then $\varphi_2 u|_{z'}=0$. If $z'$ passes through $y'_1$, then $\varphi_2 u|_{z'}=u|_z$, for a particular traversal $z$ of $(\sigma',\tau)$. Moreover, in this second case there are exactly $q$ many squares in the rightmost column that are in the count for $m(z')$, so $m(z')=q+m(z)$. Therefore we have
		\begin{align*}
			\omega'\varphi_2 u|_{\sigma,\tau}&=(-1)^{\varsigma(q)}\sum_{z'}(-1)^{m(z')}\varphi_2 u|_{z'}=(-1)^{\varsigma(q)}\sum_{z}(-1)^{m(z)+q} u|_z\\
			&=(-1)^q(-1)^{\varsigma(q)}\sum_{z}(-1)^{m(z)}u|_{z}=(-1)^q \omega u_{\sigma',\tau}=\varphi_1\omega u|_{\sigma,\tau}.\qedhere
		\end{align*}
	\end{proof}
	
	We can then replace the occurrences of $N^\bullet$ and $M^\bullet$ in Figure \ref{CommutativeDiagram1} with $\mathcal{S}^{\bullet-1}(\mathbf{E}_{\mathbf{B}\not\geq x};F_{\mathbf{B}\not\geq x})$ and $T^{\bullet-1}_{\mathbf{B}\not\geq x}$, respectively, adjusting the incoming and outgoing maps as the appropriate compositions with $\varphi^\bullet_1$ and $\varphi^\bullet_2$. In the resulting commutative diagram, the two columns are exact since, by Propositions \ref{first} and \ref{second}, the maps $\varphi^\bullet_1$ and $\varphi^\bullet_2$ are isomorphisms. The squares commute by the commutativity of the diagram from the morphism of long exact sequences and the claim.
	
	We finish the proof by induction on the size of $\mathbf{B}$. If $|\Obj\mathbf{B}|=1$, then 
	\[
	{T}^n_\xi=\mathcal{S}^0(\mathbf{B};\mathcal{S}^{n})=\prod_{x\in \mathbf{B}} \mathcal{S}^{n}(\mathbf{E}_x;F_x)=\mathcal{S}^{n}(\mathbf{E}_\xi;F_\xi),
	\]
	and $\omega=(-1)^{\varsigma(q)}\id$, so $\omega$ is a quasi-isomorphism.

	If $\omega:\mathcal{S}^{n}(\mathbf{E}_\xi;F_\xi)\to {T}^n_\xi$ is a quasi-isomorphism for $|\Obj\mathbf{B}|<i$, then we can form the commutative diagram in Figure \ref{CommutativeDiagram1} for $|\Obj \mathbf{B}|=i$, with $N^\bullet$ and $M^\bullet$ replaced as discussed above. Each row other than the middle one contains an instance of the inductive hypothesis, since both $\mathbf{B}_{\not\geq x}$ and $\mathbf{B}_{\geq x}$ have fewer objects than $\mathbf{B}$; and $\mathbf{B}$ is recursively admissible. Therefore, by the Five Lemma, the middle row is an isomorphism and thus $\omega$ is a quasi-isomorphism. This completes the induction and the proof of the theorem.
\end{proof}

\section{A reduction property for sheaf cohomology}\label{sec:application}

The statement of \ref{MainTheorem} closely resembles that of \cite[Theorem 5.1]{ET12}. Despite this, the reframing of the result in terms of sheaf cohomology, as opposed to coloured poset homology, leads to applications that are quite different from those of the coloured poset version. The key difference, explored in this section, is that while the theorem in \cite{ET12} models complex interactions between the homologies of the fibers of a bundle of coloured posets (seen in the application to Khovanov homology), the main theorem of this paper implies that if $\xi:\mathbf{B}\to\mathbf{Sh}$ is a poset bundle of sheaves with $\mathbf{B}$ recursively admissible, then it is only the cohomology of the sheaf at the maximum of $\mathbf{B}$ that contributes to the cohomology of the total sheaf of $\xi$.

By the end of this chapter, we will be able to conclude that, for example, the cohomology of a sheaf on the poset in Figure \ref{TBAPoset} can only be non-zero in degrees $0$ and $1$.

\begin{figure}[h]
	\[\begin{tikzpicture}[scale=.8]
		\node (1) at (0,5) {$\bullet$};
		\node (2) at (2,5) {$\bullet$};
		\node (3) at (0,4) {$\bullet$};
		\node (4) at (1,4) {$\bullet$};
		\node (5) at (2,4) {$\bullet$};
		\node (6) at (3,4) {$\bullet$};
		\node (7) at (4,4) {$\bullet$};
		\node (8) at (1,3) {$\bullet$};
		\node (9) at (3,3) {$\bullet$};
		\node (10) at (4,3) {$\bullet$};
		\node (11) at (0,2) {$\bullet$};
		\node (12) at (1,2) {$\bullet$};
		\node (13) at (3,2) {$\bullet$};
		\node (14) at (4,2) {$\bullet$};
		\node (15) at (0,1) {$\bullet$};
		\node (16) at (2,1) {$\bullet$};
		\node (17) at (4,1) {$\bullet$};
		\node (18) at (2,0) {$\bullet$};

		\draw[-]		(1.center)  edge (3.center) 
		(1.center)  edge (4.center) 
		(2.center)  edge (4.center) 
		(2.center)  edge (5.center) 
		(2.center)  edge (6.center) 
		(2.center)  edge (7.center) 
		(3.center)  edge (11.center) 
		(3.center)  edge (12.center) 
		(3.center)  edge (13.center) 
		(4.center)  edge (8.center) 
		(5.center)  edge (8.center) 
		(5.center)  edge (9.center) 
		(6.center)  edge (9.center) 
		(6.center)  edge (10.center) 
		(7.center)  edge (12.center) 
		(7.center)  edge (10.center) 
		(8.center)  edge (11.center) 
		(9.center)  edge (14.center) 
		(10.center)  edge (13.center) 
		(10.center)  edge (14.center) 
		(11.center)  edge (15.center) 
		(11.center)  edge (16.center) 
		(12.center)  edge (16.center) 
		(13.center)  edge (16.center) 
		(14.center)  edge (16.center) 
		(14.center)  edge (17.center) 
		(15.center)  edge (18.center) 
		(16.center)  edge (18.center) 
		(17.center)  edge (18.center) ;
		
	\end{tikzpicture}\]
	\caption{The cohomology of any sheaf on this poset is zero in all degrees $\not=0,1$. Convention is that arrows go up.}
	\label{TBAPoset}
\end{figure}

It turns out that the restriction to recursively admissible posets means that we only deal with posets with $1$. 

\begin{proposition}\label{RecursivesHaveMaxes}
	Let $\mathbf{B}$ be a recursively admissible poset. Then $\mathbf{B}$ has a unique maximum.
\end{proposition}
	This follows from the recursive definition (Definition \ref{admissible}): the poset $\mathbf{B}$ is either Boolean of rank $1$, so it has a unique maximum, or all its maximums are contained in $\mathbf{B}_{\geq x}$ for some $x\succ 0$, since $\mathbf{B}_{\geq x}^{\geq y}\not=\emptyset$ for all $y\in\mathbf{B}_{\not\geq x}$.

The admissibility property provides a kind of `factorisation' for posets into bundles. The simplest way to do this is to turn an admissible poset into a bundle over a Boolean lattice of rank $1$ $\mathbb{B}_1$. Note that Boolean lattices are recursively admissible, so we can later apply Theorem \ref{MainTheorem}.

\begin{lemma}
	Let $\mathbf{E}$ be an admissible poset for $\mathbf{E}',\mathbf{E}''$ and $(\mathbf{E},F)\in\mathbf{Sh}$. Then there is a poset bundle of sheaves $\xi:\mathbb{B}_1\to\mathbf{Sh}$ such that $(\mathbf{E}_\xi,F_\xi)=(\mathbf{E},F)$ (recall the construction of the total sheaf $(\mathbf{E}_\xi,F_\xi)$, Definition \ref{TotalSheaf}).
\end{lemma}
\begin{proof}
	We need to specify $\xi(0)$, $\xi(1)$, and $\xi(0\leq 1)$.
	\begin{itemize}[nosep]
		\item $\xi(0)=(\mathbf{E}',F)$,
		\item $\xi(1)=(\mathbf{E}'',F)$,
		\item the sheaf morphism $\gamma=\xi(0\leq 1)$ consists of a covariant functor (a poset map in this setting) $\gamma_1:\mathbf{E}'\to\mathbf{E}''$ and a natural transformation $\gamma_2:F\gamma_1\to F$:
		\begin{itemize}
			\item Let $\gamma_1(x)$ be the unique minimum of $\{y\in\mathbf{E}''\mid x\leq y\}$. Then if $x\leq x'$ in $\mathbf{E}'$, we have $\{y\in \mathbf{E}''\mid x\leq y\}\supseteq\{y\in\mathbf{E}''\mid x'\leq y\}$ and so $\gamma_1(x)\leq\gamma_1(x')$.
			\item Since $x\leq\gamma_1(x)$, we have a morphism $F(x)\leftarrow F(\gamma_1(x))$ from $(\mathbf{E},F)$. Set $\gamma_{2,x}$ to be this morphism.
		\end{itemize}
	\end{itemize}

	Remains to show that $(\mathbf{E},F)=(\mathbf{E}_\xi,F_\xi)$. It is enough to show that $\mathbf{E}=\mathbf{E}_\xi$ by the construction of $F_\xi$.
	If $x\leq y$ in $\mathbf{E}$ and either $x,y\in\mathbf{E}'$ or $x,y\in\mathbf{E}''$, then clearly $x\leq y$ in $\mathbf{E}_\xi$ (as an arrow of type a)). Suppose $x\leq y$ in $\mathbf{E}$ and $x\in\mathbf{E}'$, $y\in\mathbf{E}''$. Then $x\leq\gamma_1(x)\leq y$, so $x\leq y$ in $\mathbf{E}_\xi$.
	Conversely, the set of arrows in $\mathbf{E}_\xi$ is generated by inequalities that hold in $\mathbf{E}$. Therefore, $x\leq y$ in $\mathbf{E}$ if and only if $x\leq y$ in $\mathbf{E}_\xi$.
\end{proof}

We can also `factorise' a poset into a bundle over a more complicated base.

\begin{proposition}
	Let $\mathbf{E}$ and $\mathbf{B}$ be posets, let $(\mathbf{E},F)\in\mathbf{Sh}$, and let $\pi:\mathbf{E}\to\mathbf{B}$ be an onto poset map, such that for all $x<y$ in $\mathbf{B}$, the subposet $\pi^{-1}(x)\cup\pi^{-1}(y)$ of $\mathbf{E}$ is admissible for $\pi^{-1}(x),\pi^{-1}(y)$. Then there is a poset bundle of sheaves $\xi:\mathbf{B}\to\mathbf{Sh}$ such that $(\mathbf{E},F)=(\mathbf{E}_\xi,F_\xi)$.	
\end{proposition}

\begin{proof}
	Following the approach from the previous proposition, set $\xi(x)=(\pi^{-1}(x),F)$ and if $x<y$ in $\mathbf{B}$, then $\xi_1(x<y)$ sends $z\in\pi^{-1}(x)$ to the minimum of the subposet $\{w\in\pi^{-1}(y)\mid z\leq w\}$.
	
	Now suppose $z<w$ in $\mathbf{E}$ and $z\in\pi^{-1}(x)$, $w\in\pi^{-1}(y)$. Since $\pi$ is a poset map, $x<y$ in $\mathbf{B}$ and $z<\xi_1(x<y)(z)\leq w$ in $\mathbf{E}_\xi$.
	
	If $z<w$ in $\mathbf{E}_\xi$ is an arrow of type b) or a composition arrow, then by Proposition \ref{Composition arrows} there is a $v\in\pi^{-1}(\pi(w))$, such that $z<v<w$ in $\mathbf{E}_\xi$, where $z<v$ and $v<w$ are arrows of type b) and a), respectively. But both those arrows exist in $\mathbf{E}$, so $z<w$ in $\mathbf{E}$. 
\end{proof}

The following is a consequence of recursively admissible posets' having a unique maximum (or final object).

\begin{proposition}
	Let $\mathbf{B}$ be a recursively admissible poset and let $\xi:\mathbf{B}\to\mathbf{Sh}$ be a poset bundle of sheaves. If $1\in\mathbf{B}$ is the unique maximal object, then 
	\[
	H^\bullet(\mathbf{E}_\xi,F_\xi)\cong H^\bullet(\xi(1)).
	\]
\end{proposition}
\begin{proof}
	Let $E$ be the spectral sequence associated with $\xi$. We know that $E_2^{p,q}=H^p(\mathbf{B};\mathcal{H}_{fib}^q)$. Now, $\mathbf{B}$ has a unique maximum $1$ (Proposition \ref{RecursivesHaveMaxes}), so the functors $\varprojlim_\mathbf{B}$ and the `evaluation at $1$' functor $\_(1):\mathbf{Sh}(\mathbf{B})\to\prescript{}{R}{Mod}$ are naturally isomorphic. But we know that evaluation functors are exact. Therefore
	\[
	H^p(\mathbf{B};\mathcal{H}_{fib}^q)=\left\{\begin{array}{c c l}H^q(\xi(1)), && \text{ if }p=0, \\ 0, && \text{ otherwise}.\end{array}\right.
	\]
	
	Thus the spectral sequence collapses, we get $H^n(T^\bullet_\xi)\cong H^n(\xi(1))$, and since $\mathbf{B}$ is recursively admissible, Theorem \ref{MainTheorem} applies. This means we have $H^\bullet(\xi(1))\cong H^\bullet(T^\bullet_\xi)\cong H^\bullet(\mathbf{E}_\xi,F_\xi)$.
\end{proof}

We can now package the discussion into our main application.

\begin{theorem}\label{MainApplication}
	Let $\mathbf{E}$ and $\mathbf{B}$ be posets, with $\mathbf{B}$ recursively admissible. Suppose that \hbox{$\pi:\mathbf{E}\to\mathbf{B}$} is an onto poset map such that for all $x<y$ in $\mathbf{B}$, the subposet \hbox{$\pi^{-1}(x)\cup\pi^{-1}(y)$} of $\mathbf{E}$ is admissible for $\pi^{-1}(x),\pi^{-1}(y)$. Then
	\[
	H^\bullet(\mathbf{E};F)\cong H^\bullet(\pi^{-1}(1);F)
	\]
	for all $F\in\mathbf{Sh}(\mathbf{E})$, where $1$ is the unique maximum of $\mathbf{B}$.
\end{theorem}

\begin{remark}
	The above recipe can be applied repeatedly. Indeed, one can imagine cases where a poset $\mathbf{E}$ is admissible for $\mathbf{E}_1,\mathbf{E}_2$, and $\mathbf{E}_2$ is admissible for $\mathbf{E}_3,\mathbf{E}_4$, but $\mathbf{E}_1$ is not admissible, so the poset map $\pi:\mathbf{E}\to\mathbb{B}_2$ required for the above theorem does not exist. Despite this, we can apply the theorem twice with $\mathbf{B}=\mathbb{B}_1$ and deduce that $H^\bullet(\mathbf{E};F)\cong H^\bullet(\mathbf{E}_4;F)$, for any $F\in\mathbf{Sh}(\mathbf{E})$.
	
	Conversely, if the required poset map $\pi:\mathbf{E}\to\mathbf{B}$ exists for some recursively admissible $\mathbf{B}$, we can instead repeatedly apply Theorem \ref{MainApplication} for $\mathbb{B}_1$, at each step applying the recursive definition. The upshot is that replacing the recursively admissible $\mathbf{B}$ with the concrete $\mathbb{B}_1$ in the above theorem results in an equivalent statement.
\end{remark}

\begin{example}
	We can now examine the explicit poset given at the start of the chapter (with arrowheads omitted, but always pointing up). Let $\mathbf{E}$ be the poset in Figure \ref{TBAPoset} and choose an $F\in\mathbf{Sh}(\mathbf{E})$. First, $\mathbf{E}$ is admissible for $\mathbf{E}_1,\mathbf{E}_2$ by inspection of the left-hand side diagram in Figure \ref{BigE}. The right-hand side shows a reduction with $\mathbf{B}=\mathbb{B}_2$.
	\begin{figure}[h]\[\begin{tikzpicture}
		\node (1) at (0,5) {$\bullet$};
		\node (2) at (2,5) {$\bullet$};
		\node (3) at (0,4) {$\bullet$};
		\node (4) at (1,4) {$\bullet$};
		\node (5) at (2,4) {$\bullet$};
		\node (6) at (3,4) {$\bullet$};
		\node (7) at (4,4) {$\bullet$};
		\node (8) at (1,3) {$\bullet$};
		\node (9) at (3,3) {$\bullet$};
		\node (10) at (4,3) {$\bullet$};
		\node (11) at (0,2) {$\bullet$};
		\node (12) at (1,2) {$\bullet$};
		\node (13) at (3,2) {$\bullet$};
		\node (14) at (4,2) {$\bullet$};
		\node (15) at (0,1) {$\bullet$};
		\node (16) at (2,1) {$\bullet$};
		\node (17) at (4,1) {$\bullet$};
		\node (18) at (2,0) {$\bullet$}; 
		\node(19) at (.3,-.4){$\color{blue}\mathbf{E}_1$};
		\node(20) at (3.8,5.2){$\color{red}\mathbf{E}_2$};

		\draw[-]		(1.center)  edge (3.center) 
		(1.center)  edge (4.center) 
		(2.center)  edge (4.center) 
		(2.center)  edge (5.center) 
		(2.center)  edge (6.center) 
		(2.center)  edge (7.center) 
		(3.center)  edge (11.center) 
		(3.center)  edge (12.center) 
		(3.center)  edge (13.center) 
		(4.center)  edge (8.center) 
		(5.center)  edge (8.center) 
		(5.center)  edge (9.center) 
		(6.center)  edge (9.center) 
		(6.center)  edge (10.center) 
		(7.center)  edge (12.center) 
		(7.center)  edge (10.center) 
		(8.center)  edge (11.center) 
		(9.center)  edge (14.center) 
		(10.center)  edge (13.center) 
		(10.center)  edge (14.center) 
		(11.center)  edge (15.center) 
		(11.center)  edge (16.center) 
		(12.center)  edge (16.center) 
		(13.center)  edge (16.center) 
		(14.center)  edge (16.center) 
		(14.center)  edge (17.center) 
		(15.center)  edge (18.center) 
		(16.center)  edge (18.center) 
		(17.center)  edge (18.center);
		
		\draw [color=blue,fill=blue!30,fill opacity =.1] plot [smooth cycle, tension=1.5] coordinates {(0,.5) (0,1.4) (2,.4) (4,1.4) (4,.5) (2,-.5)};
		\draw [color=red,fill=red!30,fill opacity =.1] plot [smooth cycle, tension=1.5] coordinates {(-.3,2.3) (.2,1.55) (2,.6) (3.8, 1.55) (4.3,2.3) (4,4.5) (0,5.5)};
		
		\node (91) at (7+0,-1+5) {$\bullet$};
		\node (92) at (7+2,-1+5) {$\bullet$};
		\node (93) at (7+0,-1+4) {$\bullet$};
		\node (94) at (7+1,-1+4) {$\bullet$};
		\node (95) at (7+2,-1+4) {$\bullet$};
		\node (96) at (7+3,-1+4) {$\bullet$};
		\node (97) at (7+4,-1+4) {$\bullet$};
		\node (98) at (7+1,-1+3) {$\bullet$};
		\node (99) at (7+3,-1+3) {$\bullet$};
		\node (910) at (7+4,-1+3) {$\bullet$};
		\node (911) at (7+0,-1+2) {$\bullet$};
		\node (912) at (7+1,-1+2) {$\bullet$};
		\node (913) at (7+3,-1+2) {$\bullet$};
		\node (914) at (7+4,-1+2) {$\bullet$};
		\node (916) at (7+2,-1+1) {$\bullet$};
		\node(919) at (7+1.1,-1+1){$\color{blue}\mathbf{E}_3$};
		\node(920) at (7+3,-1+5.4){$\color{red}\mathbf{E}_6$};
		\node(921) at (7+4.7,-1+3) {$\color{green}\mathbf{E}_5$};
		\node(922) at (7-.7,-1+3) {$\color{orange}\mathbf{E}_4$};
		
		\draw[-]		(91.center)  edge (93.center) 
		(91.center)  edge (94.center) 
		(92.center)  edge (94.center) 
		(92.center)  edge (95.center) 
		(92.center)  edge (96.center) 
		(92.center)  edge (97.center) 
		(93.center)  edge (911.center) 
		(93.center)  edge (912.center) 
		(93.center)  edge (913.center) 
		(94.center)  edge (98.center) 
		(95.center)  edge (98.center) 
		(95.center)  edge (99.center) 
		(96.center)  edge (99.center) 
		(96.center)  edge (910.center) 
		(97.center)  edge (912.center) 
		(97.center)  edge (910.center) 
		(98.center)  edge (911.center) 
		(99.center)  edge (914.center) 
		(910.center)  edge (913.center) 
		(910.center)  edge (914.center) 
		(911.center)  edge (916.center) 
		(912.center)  edge (916.center) 
		(913.center)  edge (916.center) 
		(914.center)  edge (916.center) ;

		\draw [color=blue,fill=blue!30,fill opacity =.1] plot [smooth cycle, tension=1.5] coordinates {(7+2,-1+.8) (7+3.2,-1+2.1) (7+.8,-1+2.1)};
		\draw [color=green,fill=green!30,fill opacity =.1] plot [smooth cycle, tension=1.5] coordinates {(7+3.7,-1+1.9) (7+4.3,-1+1.9) (7+4.3,-1+4.1) (7+3.7,-1+4.1)};
		\draw [color=orange,fill=orange!30,fill opacity =.1] plot [smooth cycle, tension=1.5] coordinates {(7-.3,-1+1.9) (7+.3,-1+1.9) (7+.3,-1+4.1) (7-.3,-1+4.1)};
		\draw [color=red,fill=red!30,fill opacity =.1] plot [smooth cycle, tension=1.5] coordinates {(7-.2,-1+4.8) (7+.6,-1+4) (7+.6,-1+3) (7+1,-1+2.8) (7+2,-1+3.3) (7+3,-1+2.8) (7+3.4,-1+3) (7+3,-1+5) (7+0,-1+5.2)};
		
		\draw[->,tips=proper](4.7,2.5) edge (6.3,2.5);
	\end{tikzpicture}\]
	\caption{The first two reductions of the poset $\mathbf{E}$.}
	\label{BigE}
	\end{figure}
	
	Another two applications of Theorem \ref{MainApplication} with $\mathbf{B}=\mathbb{B}_1$ reduce the poset even further (Figure \ref{SmallE}).
 \begin{figure}[h]
\centering
\begin{minipage}{0.6\textwidth}
\[	\begin{tikzpicture}[scale=.8]
		\node (1) at (0,5) {$\bullet$};
		\node (2) at (2,5) {$\bullet$};
		\node (4) at (1,4) {$\bullet$};
		\node (5) at (2,4) {$\bullet$};
		\node (6) at (3,4) {$\bullet$};
		\node (8) at (1,3) {$\bullet$};
		\node (9) at (3,3) {$\bullet$};
  
		\draw[-]
		(1.center)  edge (4.center) 
		(2.center)  edge (4.center) 
		(2.center)  edge (5.center) 
		(2.center)  edge (6.center) 
		(4.center)  edge (8.center) 
		(5.center)  edge (8.center) 
		(5.center)  edge (9.center) 
		(6.center)  edge (9.center);

		\draw [color=blue,fill=blue!30,fill opacity =.1] plot [smooth cycle, tension=1.5] coordinates {(1,2.7) (1,3.4) (2,4.3) (3,3.4) (3,2.7)};
		\draw [color=red,fill=red!30,fill opacity =.1] plot [smooth cycle, tension=1.5] coordinates {(1,3.7) (2,4.7) (3.2,3.4) (3.3,3.9) (2,5.3) (-.2,5.1)};

		\node(11) at (5,4.5) {$\bullet$};
		\node(12) at (5.7,3.5) {$\bullet$};
		\node (13) at (6.4,4.5) {$\bullet$};
		\node(14) at (7.1,3.5) {$\bullet$};
		
		\draw[-]
		(11.center) edge (12.center)
		(12.center) edge (13.center)
		(13.center) edge (14.center);
		
		\draw[->,tips=proper] (7.7,4) edge (8.7,4)
		(3.6,4) edge (4.6,4);
		
		\draw[color=blue,fill=blue!30,fill opacity =.1] plot [smooth cycle, tension=1.5] coordinates {(7.3,3.7) (7.3,3.3) (6.9,3.3) (6.9,3.7)};
		\draw[color=red,fill=red!30,fill opacity =.1] plot [smooth cycle, tension=1.5] coordinates {(4.8,4.6) (5.7,3.2) (6.6,4.6)};

		\node(15) at (9.1,4.5) {$\bullet$};
		\node(16) at (9.8,3.5) {$\bullet$};
		\node(17) at (10.5,4.5) {$\bullet$};
		\draw[-]
		(15.center) edge (16.center)
		(16.center) edge (17.center);
		
		\node(18) at (10.5,3.5) {$\mathbf{E}_7$};
		\end{tikzpicture}\]
		\caption{Further reduction of the poset $\mathbf{E}$.}
		\label{SmallE}
  \end{minipage}\hfill
  \begin{minipage}{.4\textwidth}
  \[	\begin{tikzpicture}[scale=.7]
		
		\node(1) at (0,0) {$\bullet$};
		\node(2) at (.3,0) {$x$};
		\draw (0,1.6) ellipse (1 and .6);
		\draw (0,-1.6) ellipse (1 and .6);
		\draw[->,tips=proper] (-.6,-1.4) edge (1)
		(0,-1.2) edge (1)
		(.6,-1.4) edge (1)
		(1) edge (-.6,1.4)
		(1) edge (0,1.2)
		(1) edge (.6,1.4);
		\draw[color=blue,fill=blue!30,fill opacity =.1] (0,-1.6) ellipse (1.25 and .85);
		\draw[color=red,fill=red!30,fill opacity =.1] (0,1) ellipse (1.3 and 1.6);
		\node(3) at (1.9,-1.6) {$\color{blue}\mathbf{E}_{<x}$};
		\node(4) at (1.9,1) {$\color{red}\mathbf{E}_{\geq x}$};
		
	\end{tikzpicture}\]
	\caption{A decomposition of a poset with a total point.}
	\label{TotalPoint}
 \end{minipage}
	\end{figure}
		
	We thus have that $H^\bullet(\mathbf{E};F)\cong H^\bullet(\mathbf{E}_7;F)$. To see that the cohomology of $(\mathbf{E}_7,F)$ is zero for all degrees $\geq 2$, we can use the chain complex $\mathcal{T}^\bullet(\mathbf{E}_7;F):=\mathcal{S}^\bullet(\mathbf{E}_7;F)/D^\bullet$, where $D^\bullet$ is the subcomplex consisting of the degenerate simplices in $\mathbf{E}_7$, i.e.~the simplices that involve an identity arrow. This new chain complex $\mathcal{T}^\bullet$ is homotopy equivalent to $\mathcal{S}^\bullet$ (see \cite[p.138]{ET15}) and since it only involves non-degenerate simplices, its cohomology is clearly trivial at degrees $\geq 2$. 
\end{example}

There is also a more general example that we can apply our theorem to.
\begin{proposition}
	Let $\mathbf{E}$ be a poset and let $x\in\mathbf{E}$ be a total point, i.e.~for all $y\in\mathbf{E}$, either $x\leq y$ or $y\leq x$. Then $H^\bullet(\mathbf{E};F)\cong H^\bullet(\mathbf{E}_{\geq x};F)$ for any $F\in\mathbf{Sh}(\mathbf{E})$.
\end{proposition}
\begin{proof}
	If $\mathbf{E}_{<x}=\emptyset$, then $\mathbf{E}=\mathbf{E}_{\geq x}$ and the statement of the proposition is trivial. Otherwise, consider the subposets $\mathbf{E}_{\geq x}$ and $\mathbf{E}_{< x}$ (see Figure \ref{TotalPoint}). For any $y\in\mathbf{E}_{<x}$, we have $\min\mathbf{E}_{\geq x}^{\geq y}=x$ and so $\mathbf{E}$ is admissible for $\mathbf{E}_{<x},\mathbf{E}_{\geq x}$. Applying Theorem \ref{MainApplication} gives the required result.
\end{proof}

\end{document}